\documentclass{amsart}

\usepackage{amssymb}
\usepackage{amsmath}
\usepackage{amsthm}
\usepackage{bbm}
\usepackage{enumerate}
\usepackage{fancyhdr}
\usepackage{hyperref}
\usepackage{theoremref}
\usepackage{verbatim}
\usepackage{esint}
\usepackage{mathrsfs}
\usepackage{stmaryrd}

\hypersetup{
    colorlinks,
    citecolor=green,
    linkcolor=red
  }

  \makeatletter
\def\blfootnote{\xdef\@thefnmark{}\@footnotetext}
\makeatother

\newcommand{\xx}{\times}

\newcommand{\R}{\mathbb{R}}
\newcommand{\N}{\mathbb{N}}

\newcommand{\br}[1]{\left( #1 \right)}
\newcommand{\brs}[1]{\left[ #1 \right]}
\newcommand{\norm}[1]{\left\Vert #1 \right\Vert}
\newcommand{\abs}[1]{\left\vert #1 \right\vert}
\newcommand{\lb}[0]{\left\lbrace}
  \newcommand{\rb}[0]{\right\rbrace}

\DeclareMathOperator{\loc}{loc}

\DeclareMathOperator{\supp}{supp}

\newcommand\BoldSquare{%
  \setlength\fboxrule{1.1pt}\setlength\fboxsep{0pt}\fbox{\phantom{\rule{5pt}{5pt}}}}

\newenvironment{prf}{   \textsc{Proof.}}{\hfill \BoldSquare \vspace{5pt} }

\newtheorem{thm}{Theorem}[section]
\newtheorem{prop}[thm]{Proposition}
\newtheorem{lem}[thm]{Lemma}
\newtheorem{cor}[thm]{Corollary}

\theoremstyle{definition}

\newtheorem{definition}[thm]{Definition}
\newtheorem{rmk}[thm]{Remark}
\newtheorem*{notation}{Notation}

\title[Unbounded Potential Dependent Riesz Transforms]{Unboundedness of potential dependent Riesz transforms for totally irregular measures}
\author{Julian Bailey, Andrew J. Morris, Maria Carmen Reguera}

\date{}

\begin{document}

\begin{abstract}
We prove that, for totally irregular measures $\mu$ on $\R^{d}$ with $d\geq3$, the $(d-1)$-dimensional Riesz transform
$$
T_{A,\mu}^{V}f(x) = \int_{\R^d} \nabla_{1}\mathcal{E}_{A}^{V}(x,y) f(y) \, d \mu(y)
$$
adapted to the Schr\"{o}dinger operator $L_{A}^{V} = -\mathrm{div} A
\nabla + V$ with fundamental solution $\mathcal{E}_{A}^{V}$ is not
bounded on $L^{2}(\mu)$. This generalises recent results obtained by
Conde-Alonso, Mourgoglou and Tolsa for free-space elliptic operators
with H\"older continuous coefficients $A$ since it allows for the
presence of potentials $V$ in the reverse H\"{o}lder class
$RH_{d}$. We achieve this by obtaining new exponential decay estimates
for the kernel $\nabla_{1} \mathcal{E}_{A}^{V}$ as well as H\"older regularity estimates at local scales determined by the potential's critical radius function.
\end{abstract}

\maketitle


\section{Introduction}

\blfootnote{The first and third authors
are supported by the Engineering and Physical
Sciences Research Council (EPSRC) grant EP/P009239/1.}

\blfootnote{\textit{Key words and phrases.}
  Schr\"{o}dinger operator; Totally irregular measure; Riesz transforms. \\
  Mathematics Subject Classification. 42B20, 42B37.}

Suppose that $\mu$ is a Borel measure on $\R^{d}$ with $d \geq 3$. For
$0<s<d$ and suitable functions $f$, the $s$-dimensional Riesz
transforms are defined by
$$
R^{s}_{\mu}f(x):=\int_{\R^d} \frac{x-y}{|x-y|^{s+1}}f(y) \, d\mu(y).
$$
The $L^p(\mu)$-boundedness of these transforms is
deeply connected with geometric properties of the measure $\mu$
and so have been an object of study for many decades. Of particular
significance is the celebrated David--Semmes Conjecture, proven by Mattila, Melnikov and Verdera in \cite{Mattila1996} for $s=1$, and more recently by Nazarov, Tolsa and Volberg in \cite{nazarov2014uniform} for $s=d-1$. The conjecture states that when $\mu$ is an Alfhors regular 
measure, the Riesz transform $R^{s}_{\mu}$ is bounded on $L^2(\mu)$ if and only if the support of the measure $\mu$ is uniformly rectifiable. The more challenging  implication is to deduce geometric properties of the measure from Riesz transform bounds, and so one initially asks for necessary
conditions on the measure $\mu$ which imply Riesz transform bounds. Eiderman, Nazarov and Volberg considered this question in
\cite{eiderman2014s} and arrived at the conclusion that $\mu$
cannot be totally irregular. This means that it cannot happen that the upper $(d - 1)$-dimensional density
$$
\Theta^{d-1,*}(x,\mu) := \limsup_{r \rightarrow 0} \frac{\mu(B(x,r))}{r^{d-1}}
$$
is positive and finite $\mu$-almost everywhere while the lower $(d -
1)$-dimensional density
$$
\Theta^{d - 1}_{*}(x,\mu) := \liminf_{r \rightarrow 0}
\frac{\mu(B(x,r))}{r^{d - 1}}
$$
vanishes $\mu$-almost everywhere. They also investigated the case of non-integer $s$ which is geometrically distinct from the integer case. We don't consider the non-integer case here so we refer the interested reader to 
 \cite{mateu2005capacity},\cite{mateu2004riesz}, \cite{tolsa2011calderon}, \cite{reguera2016riesz} and \cite{jaye2016riesz}.

The Riesz transforms $R^{s}_{\mu}$ of dimension $s = d-1$ are also connected with boundary value problems for harmonic functions on subsets of $\R^d$. In particular,  the integral kernel $K(x,y):=(x-y)/|x-y|^{d}$ can be expressed in terms of the fundamental solution $\mathcal E^{0}_{I}(x,y):=|x-y|^{2-d}$ for the Laplace operator $\Delta$ on $\R^d$ as
\begin{equation}
  \label{eqtn:ClassicalRiesz}
K(x,y)=c_d\nabla_{1} \mathcal E^{0}_{I}(x,y),
\end{equation}
where $c_d>0$ is a dimensional constant and the notation $\nabla_{1}$ indicates that the gradient is taken with respect to the $x$-variable.

The Riesz transforms $R^{d-1}_{\mu}$ are thus inextricably linked to the Laplace operator. A natural question is then whether the definition of the Riesz transforms
$R^{d-1}_{\mu}$ can be adapted to more general elliptic differential operators and, if so, what is the relationship between geometric properties of the measure $\mu$ and  bounds for such generalised Riesz transforms. This direction of work has been pursued recently by Conde-Alonso, Mourgoglou and Tolsa in \cite{conde2019failure}, and also by Prat, Puliatti and Tolsa in \cite{prat2019boundedness}, where the theory for free-space elliptic operators with H\"older continuous coefficients has been developed.

In this paper we are primarily concerned with extending the work of Eiderman, Nazarov and Volberg, and specifically the recent work of Conde-Alonso, Mourgoglou and Tolsa in \cite{conde2019failure}, to a class of Schr\"odinger operators of the type
$$
L_{A}^{V}  = - \mathrm{div} A \nabla + V
$$
defined as unbounded operators in $L^{2}\br{\R^{d}}$ via the theory of sesquilinear forms. More precisely, we consider $d \xx d$ matrix-valued functions $A$ on $\R^d$ with real-valued coefficients that are bounded and elliptic with $\lambda, \, \Lambda > 0$ such that
\begin{equation}
  \label{eqtn:Ellipticity}
\lambda \abs{\xi}^{2} \leq \langle A(x) \xi, \xi \rangle \leq
\Lambda \abs{\xi}^{2}
\end{equation}
for all $\xi \in \R^{d}$ and all $x \in \R^{d}$, where $\langle\cdot,\cdot\rangle$ denotes the usual Euclidean inner-product. The coefficients are also assumed to be  H\"{o}lder continuous with $\alpha,\,\tau > 0$ such that
\begin{equation}
  \label{eqtn:Holder}
\abs{A(x) - A(y)} \leq \tau \abs{x - y}^{\alpha}
\end{equation}
for all $x, \, y \in \R^{d}$. The non-negative potential $V \in L^{1}_{\loc}(\R^{d})$ is assumed to be in the reverse H\"{o}lder class $RH_d$ with $C>0$ such that
$$
\br{\frac{1}{|B|}\int_{B}V^{d}}^{\frac{1}{d}} \leq C \br{\frac{1}{|B|}\int_{B} V}
$$
for all open $d$-dimensional balls $B\subset\R^d$ with Lebesgue measure $|B|$.

The $(d - 1)$-dimensional Riesz transform $T_{A,\mu}^{V}$ adapted to the operator $L_{A}^{V}$ for a Borel measure $\mu$ on $\R^{d}$ is the operator defined by 
$$
T_{A,\mu}^{V}f(x) := \int_{\R^d} \nabla_1\mathcal{E}_{A}^{V}(x,y) f(y) \, d \mu(y)
$$
for all $f \in L^{1}_{\loc}\br{\mu}$ and all $x \in \R^d \setminus \supp f$, where $\nabla_1\mathcal{E}_{A}^{V}(x,y):=\nabla(\mathcal{E}_{A}^{V}(\cdot,y))(x)$ and $\mathcal{E}_{A}^{V}$ is the fundamental solution for $L_{A}^{V}$ on $\R^{d}$. This is an extension of the classical Riesz transform, since in view of \eqref{eqtn:ClassicalRiesz} we have $R_\mu^{d-1} = c_d\, T_{A,\mu}^{V}$ when $L_{A}^{V}=\Delta$. 

The operator $T_{A,\mu}^{V}$ is said to be bounded on $L^2(\mu)$ when the set of truncated operators $T^{V}_{A,\mu,\epsilon}$, defined as above but using the kernel $\nabla_1\mathcal{E}_{A}^{V}(x,y) \mathbf{1}_{\{y\in\R^d: |x-y|>\epsilon\}}$, is uniformly bounded on $L^2(\mu)$ with respect to $\epsilon>0$. We are now in a position to state the main result of the paper.

\begin{thm} 
 \label{thm:Main} 
If the principal coefficient matrix $A$ satisfies the ellipticity and H\"older regularity in \eqref{eqtn:Ellipticity} and \eqref{eqtn:Holder} on $\R^d$, the potential $V$ is in $RH_{d}$, and the measure $\mu$ is totally irregular, then the $(d-1)$-dimensional Riesz transform $T_{A,\mu}^{V}$ adapted to the Schr\"odinger operator $L_{A}^{V}=- \mathrm{div} A \nabla + V$ is not bounded on $L^{2}(\mu)$.
\end{thm}

This result can also be interpreted from a positive viewpoint. That is, if the Riesz transform  $T_{A,\mu}^{V}$ is bounded on $L^{2}(\mu)$, then the measure $\mu$ cannot be totally irregular. The first such result for classical Riesz transforms ($A=I$, $V\equiv0$) was proved by Eiderman, Nazarov and Volberg in \cite{eiderman2014s}, whilst recently Conde-Alonso, Mourgoglou and Tolsa treated free-space elliptic operators ($V\equiv0$) with H\"older continuous coefficients in \cite{conde2019failure}.

The proof of the main result relies on orthogonality estimates via martingale differences and the David--Mattila dyadic lattice. These reduce
the problem to obtaining a local lower bound that is further simplified 
using size, smoothness and flatness properties of the operator $T_{A,\mu}^{V}$ from its kernel $\mathcal{E}_{A}^{V}$. After this reduction, we present two possible approaches. The first uses a variational argument to obtain the lower bound, following the approach in \cite{eiderman2014s} and \cite{conde2019failure}. The second uses a perturbative approach which instead reduces the analysis of the operator $L_{A}^{V}$ to the analysis of $L_{A}^{0}$ in order to apply results from the potential-free case obtained in \cite{conde2019failure}.

In both approaches, we rely on new exponential decay and H\"older regularity estimates for the kernel $\nabla_{1} \mathcal{E}_{A}^{V}$ at local scales determined by the potential's critical radius function. We also establish a flatness estimate so  that the kernel $\nabla_{1}\mathcal{E}_{A}^{V}$ can be approximated by the potential-free kernel $\nabla_{1} \mathcal{E}_{A}^{0}$. We use ideas from the work of Shen in \cite{shen1995lp} to obtain these estimates, which we find provide an interesting and valuable contribution to the theory in their own right.

The paper is organised as follows. In Section~\ref{sect:Prelim}, we detail preliminary material. In Section \ref{sec:KernelEstimates}, we obtain all of the kernel estimates required for the proof of the main result. In Section \ref{sec:Proof}, we prove Theorem~\ref{thm:Main}, beginning with the reduction to a local estimate in Section~\ref{ss:red}. We then complete the proof following the variational approach in Section~\ref{subsec:Contradiction} and then using the perturbative approach in Section \ref{sec:ShortProof}.

\section{Preliminaries}\label{sect:Prelim}

Throughout this section the coefficient matrix $A$ is assumed to
satisfy the bound and ellipticity in \eqref{eqtn:Ellipticity} on
$\R^d$ for some $d\geq3$ with constants $\lambda, \, \Lambda > 0$, as
well as the H\"{o}lder regularity in \eqref{eqtn:Holder} with
constants $\alpha, \, \tau > 0$. We begin this section by recalling
some standard regularity  results for solutions of elliptic equations
$-\mathrm{div} A \nabla u = 0$. We then provide some detail on the
properties of reverse H\"older potentials that will be required for
our analysis of Schr\"odinger equations $-\mathrm{div} A \nabla u
+ V u = 0$. Finally, we will recall the David--Mattila lattice of dyadic cubes. The following notation is introduced here for use throughout the paper.

\begin{notation}
For estimates concerning $a,\, b \in \R$, the notation $a \lesssim b$ will be used to denote that there exists a constant $C > 0$ such that $a \leq C b$. The notation $a \simeq b$ means that both $a \lesssim b$ and $b \lesssim a$ hold. The dependence of the constant $C$ on certain parameters should be clear from the context but to emphasise its dependence on a particular parameter $p\in\R$ the subscript notation $a \lesssim_{p} b$ will be used.

For sets $E,\, F$, the notation $E\subset F$ will denote that $E$ is a subset of $F$, whilst $E\subset\subset F$ will denote that the closure of $E$ is a compact subset of $F$ when $F\subset\R^d$. 

A ball $B$ in $\R^d$ will refer to an open $d$-dimensional ball $B=B(x,r):=\{y\in\R^d : |y-x| <r\}$  with centre $x\in\R^d$ and radius $r>0$, and with concentric dilates $\eta B := B(x,\eta r)$ for all $\eta>0$. Also, the Lebesgue measure of a measurable set $E\subset\R^d$ is denoted by $|E|$ and we set $\fint_E f :=  |E|^{-1} \int_E f$ for any $f\in L^1_{\loc}(\R^d)$ whenever $|E|\in(0,\infty)$.
  \end{notation}

\subsection{Divergence Form Elliptic Operators}

For non-negative $V \in L^{1}_{\loc}\br{\R^{d}}$, it is well-known (see, for
instance, {\cite[Theorem~VI.2.1]{kato1980perturbation}}) that the
bilinear form
$$
\mathfrak{l}_{A}^{V}(u,v) := \int_{\R^{d}} \langle A(x) \nabla u(x), \nabla
v(x) \rangle + \langle V(x) u(x), v(x) \rangle \, dx,
$$
defined for all $u,\, v$ in $\mathcal{V}:=\{u\in W^{1,2}\br{\R^{d}} : V^{\frac{1}{2}}u\in L^{2}\br{\R^{d}}\}$, is associated with a unique maximal accretive operator $L_{A}^{V} : D \br{L_{A}^{V}} \subset
L^{2}\br{\R^{d}} \rightarrow L^{2}\br{\R^{d}}$ such that
$$
 \int_{\R^d} \langle L_{A}^{V} u, v \rangle  = \mathfrak{l}_{A}^{V} (u,v)
$$
for all $v \in \mathcal{V} $ and all $u$ in the dense domain $D \br{L_{A}^{V}}$ given by
$$
D \br{L_{A}^{V}} = \lb u \in \mathcal{V} : \sup\nolimits_{v\in C_c^\infty(\R^d)} |\mathfrak{l}_{A}^{V}(u,v)| / \|v\|_{L^2(\R^d)} < \infty \rb.
$$
This operator has the formal expression $L_{A}^{V}=-\mathrm{div} A
\nabla + V$. We will write that $L_{A}^{V}u=0$, or $-\mathrm{div} A \nabla u + V u = 0$, in an open set $\Omega \subset \R^{d}$ when $u\in
W^{1,2}_{\loc}(\Omega)$ with $V^{\frac{1}{2}} u \in
L^{2}_{\loc}(\Omega)$ and $\int_{\Omega} ( \langle A
\nabla u, \nabla \phi \rangle + \langle V u, \phi \rangle ) = 0$ for all
smooth compactly supported functions $\phi \in C^{\infty}_{c}
\br{\Omega}$. We will also call such a function $u$ a weak solution of the equation $L_{A}^{V}u = 0$ on $\Omega$ or simply write that $u$ is $L_{A}^{V}$-harmonic on $\Omega$.

An essential part of our main argument relies on a well-known weak maximum principle. In particular, if $u \in W^{1,2}(\Omega)$ with $V^{\frac{1}{2}} u \in L^{2}(\Omega)$ is a weak solution of $L_{A}^{V}u=0$ on $\Omega$, and $u$ is continuous in a neighbourhood of the boundary $\partial \Omega$, then
\begin{equation}\label{thm:MaxPrinc}
\sup_{\Omega} \abs{u} = \sup_{\partial \Omega} \abs{u}.
\end{equation}
The same holds when $u$ is not assumed to be continuous in a neighbourhood of the boundary provided the supremum on $\partial \Omega$ is suitably interpreted. This is proved in {\cite[Theorem~8.1]{gilbarg2001elliptic}} for any non-negative $V \in L^{\infty}\br{\R^d}$ but the proof therein remains valid for any non-negative $V \in L^{1}_{\loc}\br{\R^{d}}$ whenever $V^{\frac{1}{2}} u \in L^{2}(\Omega)$.

In the potential-free case, when $V$ is identically 0, we will also use the abbreviated notation $L_{A}^{0}=-\mathrm{div} A \nabla$. The theorem below records some well-known local regularity and size estimates for weak solutions in this case. The brief proof included only serves to show the explicit dependence on the scale $R$.

\begin{thm}
  \label{thm:Interior}
Let $R>0$ and suppose that $B\subset\R^d$ is a ball of radius $r(B) \leq R$. If $L_{A}^{0} u = 0$ in an open set $\Omega\subset\R^d$ and $2B\subset\subset\Omega$, then
  \begin{equation}
    \label{eqtn:Interior1}
{|\nabla u(x)-\nabla u(y)|} \lesssim_R  \frac{1}{r(B)} \br{\frac{|x-y|}{r(B)}}^\alpha
\br{\fint_{2B} \abs{u}^{2}}^{\frac{1}{2}}
\ \text{ for all }\ x,y \in B
\end{equation}
and
\begin{equation}
  \label{eqtn:Interior2}
|\nabla u(x)| \lesssim_R \frac{1}{r(B)} \br{\fint_{2B} \abs{u}^{2}}^{\frac{1}{2}}
\ \text{ for all }\ x \in B,
\end{equation}
where both implicit constants may depend only on $d$, $\lambda$, $\Lambda$, $\tau$ and $R$. The same results hold if $2B$ is replaced by $\eta B$ for any $\eta>1$ but the implicit constants will then also depend on $\eta$.
\end{thm}

\begin{prf}  
 The first estimate \eqref{eqtn:Interior1} is a well-known result of Morrey and Campanato. For instance, by {\cite[Theorem~3.2]{giaquinta1983multiple}} and Caccioppoli's inequality, it follows that
 \[
{|\nabla u(x)-\nabla u(y)|}
    \lesssim |x-y|^\alpha \br{\int_{\tfrac{3}{2}B} \abs{\nabla u}^{2}}^{\frac{1}{2}}
    \lesssim R^{\frac{d}{2}+\alpha} \frac{1}{r(B)}\br{\frac{|x-y|}{r(B)}}^{\alpha} \br{\fint_{2B} \abs{u}^{2}}^{\frac{1}{2}}
\]
for all $x,\, y \in B$, whenever $-\mathrm{div} A \nabla u = 0$ in $2B$.
 
The second estimate \eqref{eqtn:Interior2} follows at once, since
there exists (e.g. by contradiction) a point $y_0 \in B$ such that
$\abs{\nabla u(y_0)}^2 \leq \fint_{B} \abs{\nabla u}^{2}$. Hence
\begin{align*}\begin{split}  
 \abs{\nabla u(x)} &\leq \abs{\nabla u(x) - \nabla u(y_0)} + \abs{\nabla
   u(y_0)} \\
 &\lesssim R^{\frac{d}{2}+\alpha} \frac{1}{r(B)} \br{\fint_{2B} \abs{u}^{2}}^{\frac{1}{2}} + \br{\fint_{B} \abs{\nabla u}^{2}}^{\frac{1}{2}}
\end{split}\end{align*}
for all $x \in B$ and Caccioppoli's inequality can be applied to the last term.
 \end{prf}

For non-negative $V \in L^{1}_{\loc}\br{\R^{d}}$, let $\mathcal{E}_{A}^{V}$ denote the fundamental solution to the
operator $L_{A}^{V}$ on $\R^{d}$. This is a function defined on $\lb
(x,y) \in \R^{d} \xx \R^{d} : x \neq y \rb$ with the properties that
$\mathcal{E}_{A}^{V}(\cdot,y) \in L^{1}_{\loc}\br{\R^{d}}$ and $L_{A}^{V} \mathcal{E}_{A}^{V}(\cdot,
y) = \delta_{y}$ for each $y \in \R^{d}$ in the sense that
\begin{equation}
  \label{eqtn:Fundamental}
\int \langle A(x) \nabla_1 \mathcal{E}_{A}^{V}(x,y), \nabla \phi(x)
\rangle + \langle V(x) \mathcal{E}_{A}^{V}(x,y), \phi(x) \rangle \,
dx = {\phi}(y)
\end{equation}
for any $\phi \in C^{\infty}_{c}\br{\R^{d}}$, where $\nabla_1\mathcal{E}_{A}^{V}(x,y):=\nabla(\mathcal{E}_{A}^{V}(\cdot,y))(x)$. For a detailed
construction of this object refer to \cite{davey2018fundamental}. The property
\eqref{eqtn:Fundamental} together with the
definition of weak solutions implies that for any open set $\Omega
\subset \R^{d}$ and $y \notin \Omega$ the function $x \mapsto
\mathcal{E}_{A}^{V}(x,y)$ is a weak solution to $L_{A}^{V} u = 0$ on $\Omega$. 
 
We will also use $\mathcal{E}_{A}^0$ to denote the fundamental solution on $\R^d$ in the potential-free case for $L_{A}^{0}=-\mathrm{div} A \nabla u$. The following estimate was proved for bounded domains by Gr{\"u}ter and Widman in \cite[Theorem 1.1]{gruter1982green} and for unbounded domains by Hofmann and Kim in \cite[Theorem~3.1]{hofmann2007kim}.

\begin{lem}
   \label{lem:PotentialFreeSize0} If $x, \, y \in \R^{d}$, then $\mathcal{E}_{A}^{0}(x,y) \simeq {\abs{x - y}^{-(d - 2)}}$, where the implicit constant may depend only on $d$, $\lambda$ and $\Lambda$.
   \end{lem}

   We only require the upper bound from the previous lemma in this paper. The H\"{o}lder regularity of the coefficient matrix $A$ permits the following size estimates for the derivative of the fundamental solution. 
   
\begin{lem}
  \label{lem:PotentialFreeSize}
Let $R > 0$. If $x, y \in \R^{d}$, then the following estimates hold:
\begin{enumerate}
\item
$\abs{\nabla_1 \mathcal{E}_{A}^{0}(x,y)} \lesssim_{R} {\abs{x - y}^{-(d - 1)}}$ \textrm{whenever }$\abs{x - y} \leq R$;
\item
$\abs{\nabla_1 \mathcal{E}_{A}^{0}(x,y)} \lesssim_{R} {\abs{x - y}^{-(d - 2)}}$ \textrm{whenever }$\abs{x - y} \geq R$.
\end{enumerate}
The implicit constants in both cases may depend only on $d$, $\lambda$, $\Lambda$, $\tau$ and $R$.
\end{lem}

\begin{prf}  
A weaker version of the second estimate was
proved in \cite{conde2019failure}, which we modify here. If $x, \, y \in \R^{d}$ and $\abs{x - y} \geq R$, then \eqref{eqtn:Interior2} in Theorem \ref{thm:Interior} and Lemma~\ref{lem:PotentialFreeSize0} show that
 \begin{align*}\begin{split}  
 \abs{\nabla_1 \mathcal{E}_{A}^{0}(x,y)} &\leq \norm{\nabla_1
   \mathcal{E}_{A}^{0}(\cdot, y)}_{L^{\infty}(B(x,R/4))} \\
 &\lesssim_R \frac{4}{R}\norm{\mathcal{E}_{A}^{0}(\cdot,y)}_{L^{\infty}(B(x,R/2))} \\
 &\lesssim_R \sup_{\tilde{x} \in B(x,R/2)} \frac{1}{\abs{\tilde{x} -
     y}^{d - 2}} \\
  &\lesssim \frac{1}{\abs{x - y}^{d -
  2}},
\end{split}\end{align*}
where the final inequality uses that $\abs{x - y} \leq 2
\abs{\tilde{x} - y}$ for all $\tilde{x} \in B(x,R/2)$.

A proof of the first estimate can be found in  \cite{kenig2011layer} but note that if $\abs{x - y} \leq R$, then applying the argument above on the ball $B(x,\abs{x - y}/4)$ gives the result.
\end{prf}

It is interesting to note that the previous size estimates for
$\nabla_{1}\mathcal{E}_{A}^{0}$ sharply contrast to the case $A =
I$, since they do not bound the derivative universally from above by a multiple of $\abs{x - y}^{-(d -
     1)}$. Instead, owing to the perturbation $A$, the previous
   estimates assert  a weaker rate of global decay for $\nabla_{1} \mathcal{E}_{A}^{0}$.

\subsection{Reverse H\"{o}lder Potentials}
\label{subsec:Reverse}

A non-negative function $V \in L^{1}_{\loc}(\R^{d})$ is said to belong to
the reverse H\"{o}lder class $RH_{q}$ of index $q \in
(1,\infty)$ when
$$
\llbracket V \rrbracket_{q} := \sup_{B\subset\R^d} \br{\fint_{B}V^{q}}^{\frac{1}{q}} \br{\fint_{B} V}^{-1} < \infty,
$$
where the supremum is taken over all open $d$-dimensional balls $B$ in $\R^d$.

We now recall some fundamental properties of reverse H\"{o}lder potentials that will be used throughout the paper. First, recall that reverse H\"{o}lder potentials are a source of doubling measures, whereby if $V \in RH_{q}$
for some $q \in (1,\infty)$, then
\begin{equation}
  \label{eqtn:0.3}
  V(B(x,2r)) \lesssim V(B(x,r))
\end{equation}
for all $x \in \R^{d}$ and $r > 0$, where $V(E) := \int_{E} V(y) \, dy $ for measurable sets $E \subset \R^{d}$. This follows from the fact that
reverse H\"{o}lder potentials are
$A_{\infty}$-weights, which are doubling. These facts and the following well-known self-improvement property can be found, for instance, in \cite[Chapter 9]{grafakos2009modern}.

\begin{prop}
  \label{prop:SelfImprovement}
  If $V \in RH_{q}$ for some $q \in (1,\infty)$, then there
   exists $\varepsilon > 0$ such that $V \in RH_{q'}$ for all  $q' \in [q,q + \varepsilon)$.
\end{prop}

We also need the following lemma proved by Shen in \cite{shen1995lp} which  quantifies how the measure $V(B(x,r))$ decreases as $r$  decreases.

\begin{lem}[{\cite[Lemma~1.2]{shen1995lp}}]
  \label{lem:RHn}
If $V \in RH_{q}$ for some $q \in (1,\infty)$, then
$$
V (B(x,r)) \lesssim \br{\frac{r}{R}}^{d - \frac{d}{q}} V \br{B(x,R)}
$$
for all $x \in \R^{d}$ and $0 < r < R$, where the implicit constant depends only on $V$ through $\llbracket V \rrbracket_{q}$.
\end{lem}

The subsequent lemma will be used numerous times when we come to
prove estimates for the kernel $\nabla_{1} \mathcal{E}^{V}_{A}$.

\begin{lem}
  \label{lem:RHn2}
The following estimates hold for all $x \in \R^{d}$ and $r > 0$:\\ If $V \in RH_{\frac{d}{2}}$, then
   \begin{equation}
    \label{eqtn:0.4}
\int_{B(x,r)} \frac{V(y)}{\abs{y - x}^{d - 2}} \, dy \lesssim
\frac{V(B(x,r))}{r^{d - 2}};
\end{equation}
If $V \in RH_{d}$, then
  \begin{equation}
    \label{eqtn:RHn2}
\int_{B(x,r)} \frac{V(y)}{\abs{y - x}^{d - 1}} \, dy \lesssim
\frac{V(B(x,r))}{r^{d - 1}}.
\end{equation}
The implicit constants in
\eqref{eqtn:0.4} and \eqref{eqtn:RHn2} depend only on $V$ through
$\llbracket V \rrbracket_{\frac{d}{2}}$ and $\llbracket V
\rrbracket_{d}$ respectively.
\end{lem}

\begin{prf}  
  Let's first consider the second estimate. Suppose that $V \in RH_{d}$. On splitting the integral in \eqref{eqtn:RHn2} into annuli,
  \begin{align*}\begin{split}  
 \int_{B(x,r)} \frac{V(y)}{\abs{y - x}^{d - 1}} \, dy &= \sum_{k =
   0}^{\infty} \int_{B(x,2^{-k}r) \setminus B(x,2^{-(k + 1)}r)}
 \frac{V(y)}{\abs{y - x}^{d - 1}} \, dy \\
 &\leq \sum_{k = 0}^{\infty} \frac{V(B(x,2^{-k} r))}{(2^{-(k +
     1)}r)^{d - 1}}.
\end{split}\end{align*}
Since $V \in RH_{d}$, it follows from the self-improvement property
that $V \in RH_{q'}$ for some $q' > d$. Applying Lemma \ref{lem:RHn} gives
\begin{align*}\begin{split}  
  \int_{B(x,r)} \frac{V(y)}{\abs{y - x}^{d - 1}} \, dy &\leq \sum_{k =
  0}^{\infty} \br{2^{-k}}^{d - \frac{d}{q'}} \frac{V(B(x,r))}{(2^{-(k +
    1)}r)^{d - 1}} \\
&\lesssim \br{ \sum_{k = 0}^{\infty} 2^{-k \br{1 - \frac{d}{q'}}}}
\frac{V(B(x,r))}{r^{d - 1}} \\
&\lesssim \frac{V(B(x,r))}{r^{d - 1}}.
\end{split}\end{align*}
The first estimate can be proved in an identical manner.
\end{prf}

\subsection{Schr\"{o}dinger Operators}
\label{subsec:Schrodinger}

Let $V \in RH_{\frac{d}{2}}$ and consider the Schr\"{o}dinger operator $
L_{I}^{V} = - \Delta + V.
$
In the paper \cite{shen1995lp},
Shen introduced technical machinery that could be used for
the analysis of Schr\"{o}dinger operators with potentials in the
reverse H\"{o}lder class $RH_{\frac{d}{2}}$.
As the proof of our result will rely heavily on this machinery, it
will be fruitful to recall any
pertinent details and state any result that will be
required for the proof of our theorem.

At the heart of Shen's Schr\"{o}dinger operator machinery is the
critical radius function. This is the function $\rho : \R^{d}
\rightarrow [0,\infty)$ defined through
\begin{equation}
  \label{eqtn:CritRad}
\rho(x) := \sup \lb r > 0 : \frac{1}{r^{d - 2}}
\int_{B(x,r)} V(x) \, dx \leq 1 \rb
\end{equation}
for $x \in \R^{d}$. The physical intuition and drive behind the
introduction of this function is that if the potential does not, on
average, oscillate too wildly, then the Schr\"{o}dinger operator
should behave locally like the classical Laplacian $-\Delta$. The
function $\rho$ precisely determines this local scale.

\begin{rmk}
  \label{rmk:Critical}
It follows directly from the definition of the critical radius
function that
$
V(B(x,\rho(x))) \simeq \rho(x)^{d - 2}
$
for all $x \in \R^{d}$.
\end{rmk}

The following lemma proved by Shen in \cite{shen1995lp} allows us to compare the
critical radius function at two distinct points.

\begin{lem}[{\cite[Lemma~1.4]{shen1995lp}}]
  \label{lem:Shen0}
If $V \in RH_{\frac{d}{2}}$, then there exist $B_{0}, M_{0} > 0$, depending only on $V$ through $\llbracket V \rrbracket_{\frac{d}{2}}$, such that
 \[
        B_{0}^{-1} \rho(x) \br{1 + \frac{\abs{x -
          y}}{\rho(x)}}^{-M_{0}} \leq \rho(y) \leq B_{0}
    \rho(x) \br{1 + \frac{\abs{x -
          y}}{\rho(x)}}^{\frac{M_{0}}{M_{0} + 1}}
   \]
     for all $x, \, y \in \R^{d}$.
  \end{lem}

An immediate corollary of this lemma is that the critical radius
function will be bounded from below on compactly supported sets. If, in addition, $V \neq 0$
then it will also be bounded from above.

\begin{cor}
  \label{cor:Shen0}
  Suppose that $V \in RH_{\frac{d}{2}}$. If $M > 0$ and $E \subset B(0,M) \subset \R^{d}$, then
  $$
B_{0}^{-1} \rho(0) \br{1 + \frac{M}{\rho(0)}}^{-M_{0}} \leq
\rho(y) \leq B_{0} \rho(0) \br{1 +
  \frac{M}{\rho(0)}}^{\frac{M_{0}}{M_{0} + 1}}
$$
for all $y \in E$, where $B_{0}$ and $M_{0}$ denote the constants from Lemma \ref{lem:Shen0}.
\end{cor}

Another straightforward corollary tells us that the critical
radius function does not tend to vary too much at a local level.

\begin{cor}
\label{cor:LocalCritical}
Suppose that $V \in RH_{\frac{d}{2}}$. If $\eta>0$ and $x\in\R^d$, then $$\rho(x) \simeq_\eta \rho(y) \quad \text{ for all } \quad y \in B(x,\eta\rho(x)),$$ where the implicit constants may depend only on $d$, $\llbracket V\rrbracket_{d/2}$ and $\eta$. 
\end{cor}

 Let $d_{V}(x,y)$ denote the Agmon distance for the
potential defined by
$$
d_{V}(x,y) := \inf_{\gamma} \int^{1}_{0} \rho(\gamma(t))^{-1}
\abs{\gamma'(t)} \, dt,
$$
where the infimum is taken over all curves in $\R^d$
connecting the points $x, \, y \in \R^{d}$. In the paper \cite{shen1999fundamental}, Shen obtained sharp estimates for the
fundamental solution of the Schr\"{o}dinger operator expressed in
terms of the Agmon distance. More recently, Mayboroda and Poggi
in \cite{mayboroda2019exponential}
have generalised these sharp estimates to the operator
$L_{A}^{V}$. These estimates are stated in the  theorem below. It
should be noted that only the upper estimate will be used in this paper.

\begin{thm}[{\cite[Corollary~6.16]{mayboroda2019exponential}}] 
 \label{thm:MaybSharp} 
If $V \in RH_{\frac{d}{2}}$, then there exist $\varepsilon, \, \varepsilon' > 0$ such that 
 $$
\frac{e^{-\varepsilon' d_{V}(x,y)}}{\abs{x - y}^{d - 2}} \lesssim
\mathcal{E}_{A}^{V}(x,y) \lesssim \frac{e^{-\varepsilon d_{V}(x,y)}}{\abs{x -
  y}^{d - 2}}
   \quad \text{ for all }x, \, y \in \R^{d},
$$
where $\varepsilon$, $\varepsilon'$ and the implicit constants may depend only on $d$, $\lambda$, $\Lambda$ and
$\llbracket V \rrbracket_{\frac{d}{2}}$.
\end{thm}

 As the Agmon distance is currently defined, it is difficult to
 discern how the distance will vary for a particular potential.
 The  lemma below will demystify the Agmon distance by comparing
 it with the quantity $\br{1 + \frac{\abs{x - y}}{\rho(x)}}$.
 
\begin{lem}[{\cite[Remark 3.21]{shen1999fundamental}}]
  \label{lem:Shen}
If $V \in RH_{\frac{d}{2}}$,  then
  \begin{eqnarray}
    \label{eqtn:lem:Shen1}
    &\displaystyle &d_{V}(x,y) \lesssim \br{1 + \frac{\abs{x - y}}{\rho(x)}}^{M_{0} + 1}  \textrm{ for all }\, x, \, y \in \R^{d},\\  \!\!\!\!\!\!\!\!\textrm{whilst} \nonumber \\
     \label{eqtn:lem:Shen2}
    &\displaystyle &d_{V}(x,y) \gtrsim  \br{1 + \frac{\abs{x - y}}{\rho(x)}}^{\frac{1}{M_{0} + 1}} \textrm{ for all }\, x, \, y \in \R^{d} \textrm{ when } \abs{x - y} \geq 
  \rho(x),
  \end{eqnarray}
where $M_{0}$ denotes the constant from Lemma~ \ref{lem:Shen0}, and the implicit constant in each estimate depends only on $d$ and $\llbracket V \rrbracket_{\frac{d}{2}}$.
\end{lem}

\subsection{The David--Mattila Lattice}

 In the paper
\cite{david2000removable}, David and Mattila
introduced a system of cubes on the support of the measure $\mu$ that
were analagous to the standard
dyadic cubes on $\R^{d}$. Let $C_{0}, \, A > 1$  with $A > 5000 C_{0}$. These
constants will be the parameters of the lattice.

\begin{thm}[{\cite[Theorem~3.2]{david2000removable}}]
  \label{thm:DavidMattila}
  There exists a sequence of partitions of $\supp  \mu$,
  $\mathcal{D} = \cup_{k \geq k_{0}} \mathcal{D}_{k}$, into Borel subsets $Q$
  with the following properties.
  \begin{enumerate}
    \item[$\bullet$]  $\mathcal{D}_{k_{0}} = \lb Q_{k_{0}} \rb$ where $Q_{k_{0}} :=
  \supp  \mu$.
\item[$\bullet$] For each $k \geq k_{0}$, $\supp  \mu$ is
  the disjoint union of the sets $Q \in \mathcal{D}_{k}$.  
  \item[$\bullet$] If $k_{0} \leq k < l$, $Q \in \mathcal{D}_{k}$  and
    $R \in \mathcal{D}_{l}$ then either $Q \cap R = \emptyset$ or $R
    \subset Q$.
\item[$\bullet$] For each $k \geq k_{0}$ and $Q \in \mathcal{D}_{k}$,
  there is a ball $B(Q) := B(x_{Q}, r(Q))$ with $x_{Q} \in
  \supp  \mu$,
  $$
A^{-k} \leq r(Q) \leq C_{0} A^{-k},
$$
$$
\supp  \mu \cap B(Q) \subset Q \subset \supp  \mu
\cap 28 B(Q) = \supp  \mu \cap B(x_{Q},28 r(Q))
$$
and the balls $5 B(Q)$ for $Q \in \mathcal{D}_{k}$ are all disjoint.

\item[$\bullet$] The cubes $Q \in \mathcal{D}_{k}$ have small
  boundaries in the following sense. For each $Q \in \mathcal{D}_{k}$ and $l
  \geq 0$ set
  $$
N^{ext}_{l}(Q) := \lb x \in \supp  \mu  \setminus Q :
\mathrm{dist}(x,Q) < A^{-k-l} \rb,
$$
$$
N^{int}_{l}(Q) := \lb x \in Q : \mathrm{dist}(x, \supp  \mu
\setminus Q) < A^{-k-l} \rb
$$
and
$$
N_{l}(Q) := N^{ext}_{l}(Q) \cup N^{int}_{l}(Q).
$$
Then
$$
\mu \br{N_{l}(Q)} \leq \br{C^{-1} C_{0}^{-3d - 1} A}^{-l} \mu(90 B(Q)),
$$
where the constant $C > 0$ depends only on the dimension $d$.
\item[$\bullet$] Let $\mathcal{D}_{k}^{db}$ denote the set of cubes $Q
  \in \mathcal{D}_{k}$ for which
  $$
\mu(100 B(Q)) \leq C_{0} \mu(B(Q)).
$$
When $Q \in \mathcal{D}_{k} \setminus \mathcal{D}_{k}^{db}$ we have
that $r(Q) = A_{0}^{-k}$ and
$$
\mu \br{100 B(Q)} \leq C_{0}^{-l} \mu\br{100^{l + 1}B(Q)}
$$
for all $l \geq 1$ with $100^{l} \leq C_{0}$.
    \end{enumerate}
  \end{thm}

  For $Q \in \mathcal{D}_{k}$, we define the length of $Q$ to be the
  quantity
  $$
l(Q) := 56 C_{0} A_{0}^{-k}.
$$
The point $x_{Q}$ is called the center of the
cube $Q$. The unique $Q' \in \mathcal{D}_{k - 1}$ such that $Q \subset
Q'$ is called the parent of the cube $Q$. Similarly, the cubes $R \in
\mathcal{D}_{k + 1}$ for which $R \subset Q$ will be called the
children of $Q$. We will use the shorthand notation $B_{Q} := 28 B(Q)$.

Set 
$$
\mathcal{D}^{db} := \bigcup_{k \geq k_{0}} \mathcal{D}_{k}^{db}.
$$
For $Q \in \mathcal{D}$ let $\mathcal{D}(Q)$ denote the cubes in
$\mathcal{D}$ that are contained in $Q$. Also let $\mathcal{D}^{db}(Q)
:= \mathcal{D}^{db} \cap \mathcal{D}(Q)$. From this point onwards,
$A_{0}$ will be set large enough so that it satisfies
$$
C^{-1} C_{0}^{-3 d - 1} A_{0} > A_{0}^{\frac{1}{2}} > 10.
$$
It will then follow that for any $0 < \lambda \leq 1$,
\begin{align}\begin{split}  
    \label{eqtn:SmallBoundary}
      &\mu(\lb x \in Q : \mathrm{dist}(x, \supp \mu \setminus Q)
  \leq \lambda l(Q) \rb) \\ & \quad + \mu(\lb x \in 3.5 B_{Q} \setminus Q :
  \mathrm{dist}(x,Q) \leq \lambda l(Q) \rb) \leq c
  \lambda^{\frac{1}{2}} \mu(3.5 B_{Q}).
 \end{split}\end{align}

 \section{Kernel Estimates}
 \label{sec:KernelEstimates}

Throughout this section we continue to consider a coefficient matrix $A$ on $\R^d$ for some $d\geq3$ satisfying the assumptions in Section \ref{sect:Prelim} with constants $\lambda, \, \Lambda,\, \alpha, \, \tau > 0$. We also assume that $V$ is a fixed non-negative locally integrable function on $\R^d$ that belongs to the reverse H\"older class $RH_d$. In particular, the self-improvement property in Proposition~\ref{prop:SelfImprovement} implies that there exists $\delta>0$ such that $V \in RH_{d+\delta}$ and the constant $\beta := 2 - \frac{d}{d+\delta}$ is henceforth fixed. Note that $\beta \in (1,2)$ whilst by Lemma~\ref{lem:RHn} there is  the following volume comparison estimate
\begin{equation}
  \label{eqtn:beta}
  V(B(x,r)) \lesssim \br{\frac{r}{R}}^{d - 2 + \beta}V(B(x,R))
\end{equation}
for all $x \in \R^{d}$ and $0< r < R$. To simplify notation, we now set
$$
L:=L_{A}^{V}
\qquad\textrm{and}\qquad
\mathcal{E}:=\mathcal{E}_{A}^{V},
$$
whilst  $L_{A}^{0}$ and $\mathcal{E}_{A}^{0}$ denote the case when $V$ is identically 0.

Our first estimate enables us to deduce exponential decay estimates for the kernel $\nabla_{1} \mathcal{E}$ from those for $\mathcal{E}$ in Theorem~\ref{thm:MaybSharp} obtained by Mayboroda and Poggi in \cite{mayboroda2019exponential}. The proof below relies on ideas from Shen's work in \cite{shen1999fundamental}.

  \begin{prop}
    \label{prop:Size}
    Let $R>0$. If $x, \, y\in \R^{d}$ and $0<r \leq \min\{R,\abs{x - y}\}$, then
    \begin{equation}
      \label{eqtn:Size}
\abs{\nabla_{1} \mathcal{E}(x,y)} \lesssim_{R} 
\frac{1}{r} \|\mathcal{E}(\cdot,y)\|_{L^\infty(B(x,r/2))} \br{1 +
  \frac{V(B(x,r))}{r^{d - 2}}},
\end{equation}
where the implicit constant depends only on $d$, $\lambda$, $\Lambda$, $\tau$, $\llbracket V\rrbracket_{d/2}$ and $R$.
\end{prop}

\begin{prf}  
 Fix $x, \, y \in \R^{d}$. Let $R>0$, suppose that $0 < r \leq \min\{R,\abs{x - y}\}$ and consider the ball $B:= B(x,r/2)$. Define
\begin{equation}\label{defu}
u(\xi) := \mathcal{E}(\xi,y)
\quad \textrm{ for all }\quad \xi \in B
\end{equation}
and note that $-\mathrm{div}A\nabla u + V u = 0$ in $B$. Next, define
\begin{equation}\label{defv}
v(\xi) := u(\xi) + \int_{B} \mathcal{E}_{A}^{0}(\xi,\zeta) u(\zeta) V(\zeta) \, d\zeta
\quad \textrm{ for all }\quad \xi \in B
\end{equation}
and note that the arguments on page 537 in the proof of Lemma~2.20 in \cite{shen1999fundamental} show that $-\mathrm{div}A\nabla v = 0$ in $B$. In particular, the proof that $v\in W^{1,2}(B)$ relies on the first estimate in Lemma \ref{lem:PotentialFreeSize} whilst Fubini's Theorem and \eqref{eqtn:Fundamental} show that
\begin{align*}
\int_B \langle A \nabla v, \nabla \phi \rangle 
&=\int_B \langle A \nabla u, \nabla \phi \rangle
+ \int_B \br{\int_B \langle A(\xi) \nabla_1 \mathcal{E}_{A}^{0}(\xi,\zeta), \nabla \phi(\xi) \rangle \, d\xi} u(\zeta) V(\zeta) \, d\zeta \\
&=\int_B \langle A \nabla u, \nabla \phi \rangle + \int_B \langle V u
                                                                                                                                               ,
                                                                                                                                               \phi
                                                                                                                                               \rangle
                                                                                                                                               = 0
\end{align*}
for all $\phi\in C_c^\infty(B)$.

The size estimate \eqref{eqtn:Interior2} in Theorem \ref{thm:Interior} applied to $v$ and the ball $B$ implies that
\begin{align}\begin{split}
 \label{eqtn:Size1}
    \abs{\nabla_{1} \mathcal{E}(x,y)} 
    &\leq \norm{\nabla u}_{L^{\infty}(\frac{1}{2}B)} \\
    &\leq \norm{\nabla v}_{L^{\infty}(\frac{1}{2}B)} + \sup_{\xi \in B} \int_{B} \abs{\nabla_{1}
   \mathcal{E}_{A}^{0}(\xi,\zeta)} \abs{u(\zeta)} V(\zeta) \, d\zeta \\
 &\lesssim_R \frac{1}{r} \norm{v}_{L^{\infty}(B)} +
 \norm{u}_{L^{\infty}(B)} \sup_{\xi \in B} \int_{B} \abs{\nabla_{1} \mathcal{E}_{A}^{0}(\xi,\zeta)} V(\zeta) \, d\zeta.
\end{split}\end{align}
If $\xi \in B(x,r/2)$, then since $r\leq R$ and $B(x,r/2) \subset B(\xi,r)\subset B(x,3r/2)$, the first estimate in Lemma \ref{lem:PotentialFreeSize} implies that 
\begin{align}\begin{split}
    \label{eqtn:Size2}
 \int_{B(x,r/2)} \abs{\nabla_{1} \mathcal{E}_{A}^{0}(\xi,\zeta)} V(\zeta) \, d\zeta 
\lesssim_R \int_{B(\xi,r)} \frac{V(\zeta)}{\abs{\xi - \zeta}^{d - 1}} \, d\zeta
\lesssim \frac{V(B(x,r))}{r^{d - 1}},
\end{split}\end{align}
where the second estimate uses \eqref{eqtn:RHn2} from Lemma \ref{lem:RHn2}, which requires $V\in RH_{d}$, and the doubling property \eqref{eqtn:0.3}. Estimates \eqref{eqtn:Size1} and \eqref{eqtn:Size2} combined show that
$$
\abs{\nabla_{1} \mathcal{E}(x,y)} \lesssim_R  \frac{1}{r}
\norm{v}_{L^{\infty}(B)} + \norm{u}_{L^{\infty}(B)}
\frac{V(B(x,r))}{r^{d - 1}}.
$$

To estimate $\norm{v}_{L^{\infty}(B)}$, we use Lemma \ref{lem:PotentialFreeSize0} to obtain
\begin{align*}\begin{split}
 \norm{v}_{L^{\infty}(B)} &\lesssim \norm{u}_{L^{\infty}(B)} \br{1 + \sup_{\xi \in B} \int_{B} \abs{\mathcal{E}_{A}^{0}(\xi,\zeta)} V(\zeta) \, d\zeta} \\
   &\lesssim \norm{u}_{L^{\infty}(B)} \br{1 + \sup_{\xi \in B} \int_{B} \frac{V(\zeta)}{\abs{\xi - \zeta}^{d - 2}} \, d\zeta} \\
 &\lesssim \norm{u}_{L^{\infty}(B)} \br{1 + \frac{V(B(x,r))}{r^{d - 2}}},
 \end{split}\end{align*}
where the last estimate is similar to \eqref{eqtn:Size2} except \eqref{eqtn:0.4}, which only requires $V\in RH_{\frac{d}{2}}$, is used instead of \eqref{eqtn:RHn2}. Altogether, we have
\[
\abs{\nabla_{1} \mathcal{E}(x,y)} \lesssim_{R} 
\frac{1}{r} \|u\|_{L^\infty(B(x,r/2))} \br{1 +
  \frac{V(B(x,r))}{r^{d - 2}}},
\]
which proves \eqref{eqtn:Size}.
 \end{prf}
  
We now deduce the aforementioned size estimates for the kernel $ \nabla_{1} \mathcal{E}$.

\begin{thm}
\label{thm:Size}
Let $R > 0$ and let $\varepsilon > 0$ denote the constant from Theorem \ref{thm:MaybSharp}. If $x$, $y \in \R^{d}$, then the following estimates hold:
\begin{enumerate}
\item $\displaystyle \abs{\nabla_{1} \mathcal{E}(x,y)} \lesssim_{R} {e^{-(\varepsilon/2) d_{V}(x,y)}}{\abs{x - y}^{-(d - 1)}}$ \textrm{whenever }$\abs{x - y} \leq R$;
\item $\displaystyle \abs{\nabla_{1} \mathcal{E}(x,y)} \lesssim_{R} {e^{-(\varepsilon/2) d_{V}(x, y)}}{\abs{x - y}^{-(d - 2)}}$ \textrm{whenever }$\abs{x - y} \geq R$.
\end{enumerate}
The implicit constants in both cases may depend only on $d$, $\lambda$, $\Lambda$, $\tau$, $\llbracket V\rrbracket_{d/2}$ and $R$.
\end{thm}

\begin{prf}
Let $R>0$ and fix $x, \, y \in \R^{d}$. To prove the first estimate, suppose that $\abs{x - y} \leq R$ and consider the following two cases:

\textit{Case 1.1}: Assume that $\rho(x) \leq \abs{x - y}$. Applying Proposition~\ref{prop:Size} in the case $r=\rho(x)$ followed by Remark \ref{rmk:Critical} and Theorem \ref{thm:MaybSharp} shows that
 \begin{align}\begin{split}
     \label{eqtn:thm:Size0}
  \abs{\nabla_{1} \mathcal{E}(x,y)} 
  &\lesssim_R \frac{1}{\rho(x)}
 \|{\mathcal{E}(\cdot,y)}\|_{B(x,\rho(x)/2)} \br{1 +
   \frac{V(B(x,\rho(x)))}{\rho(x)^{d - 2}}} \\
 &\lesssim \frac{1}{\rho(x)} \sup_{\xi \in B(x, \rho(x)/2)}
 \frac{e^{- \varepsilon d_{V}(\xi,y)}}{\abs{\xi - y}^{d - 2}},
\end{split}\end{align}
where $\varepsilon > 0$ is the constant from Theorem \ref{thm:MaybSharp}.

If $\xi \in B(x,\rho(x)/2)$, then $\abs{x - y} \leq 2 \abs{\xi - y}$
and $d_{V}(x, y) \leq d_{V}(\xi,y) + C$ for some $C > 0$, since
$
\abs{x - y} \leq \abs{\xi - y} + \rho(x)/2 \leq \abs{\xi - y} + {\abs{x - y}}/{2}
$
whilst $d_{V}(\xi,x) \lesssim 1$ by \eqref{eqtn:lem:Shen1} in Lemma \ref{lem:Shen}. Using these estimates in conjunction with \eqref{eqtn:thm:Size0} and  then applying \eqref{eqtn:lem:Shen2} in Lemma \ref{lem:Shen}, which is valid because $\rho(x) \leq \abs{x - y}$, shows that
\begin{align}\label{eq:repeat}\begin{split}
 \abs{\nabla_{1} \mathcal{E}(x,y)} &\lesssim_R \frac{1}{\rho(x)}
\frac{e^{- \varepsilon d_{V}(x,y)}}{\abs{x- y}^{d -
    2}} \\
&\lesssim d_{V}(x,y)^{M_{0} + 1} \frac{e^{- \varepsilon
     d_{V}(x,y)}}{\abs{x - y}^{d - 1}} \\
 &\lesssim \frac{e^{-(\varepsilon/2)  d_{V}(x,y)}}{\abs{x -
     y}^{d - 1}},
\end{split}\end{align}
where $M_{0} > 0$ is the constant from Lemma \ref{lem:Shen}, as required.

\vspace*{0.1in}

\textit{Case 1.2}: Assume that $\rho(x) \geq \abs{x - y}$. Applying Proposition~\ref{prop:Size} in the case $r = \abs{x - y}$ shows that
$$
 \abs{\nabla_{1} \mathcal{E}(x,y)} \lesssim_R \frac{1}{\abs{x -
     y}} \|\mathcal{E}(\cdot,y)\|_{L^\infty(B(x,r/2))} \br{1
   + \frac{V(B(x, \abs{x - y}))}{\abs{x - y}^{d - 2}}}.
$$
Lemma \ref{lem:RHn} with $q=d/2$ and Remark \ref{rmk:Critical} show that
$$
\frac{V(B(x,\abs{x - y}))}{\abs{x - y}^{d - 2}}
\lesssim \frac{V(B(x, \rho(x)))}{\rho(x)^{d - 2}} \lesssim 1.
$$
Also, if $\xi \in B(x,r/2)$, then $\abs{x - y} \leq 2\abs{\xi - y}$
and $d_{V}(x,y) \leq d_{V}(\xi,y) + C$ for some $C > 0$, since $\abs{x - y} \leq \abs{\xi - y} + \abs{\xi - x} \leq \abs{\xi - y}  + {\abs{x - y}}/{2}$ whilst
$d_{V}(\xi,x) \lesssim 1$ by \eqref{eqtn:lem:Shen1} in Lemma \ref{lem:Shen}.
Altogether, these estimates and Theorem \ref{thm:MaybSharp} show that
$$
\abs{\nabla_{1} \mathcal{E}(x,y)} \lesssim_R \frac{e^{- \varepsilon d_{V}(x,y)}}{\abs{x -
    y}^{d - 1}},
$$
which concludes the proof of the first estimate stated in the theorem.

\vspace*{0.1in}

To prove the remaining estimate, now suppose that $\abs{x - y} \geq R$ and consider the following two cases:

\textit{Case 2.1}: Assume that $\rho(x) \leq R$. Applying Proposition~\ref{prop:Size} once again in the case $r= \rho(x)$ shows that \eqref{eqtn:thm:Size0} holds. Moreover, since $\rho(x) \leq R \leq \abs{x - y}$, the arguments from Case 1.1 apply here to show that \eqref{eq:repeat} holds, hence
\[
\abs{\nabla_{1} \mathcal{E}(x,y)}
\lesssim_R \frac{e^{-(\varepsilon/2) d_{V}(x,y)}}{\abs{x - y}^{d - 1}}
\lesssim_R \frac{e^{- (\varepsilon/2) d_{V}(x,y)}}{\abs{x - y}^{d - 2}},
\]
as required.

\vspace*{0.1in}

\textit{Case 2.2}: Assume that $\rho(x) \geq R$. Applying Proposition~\ref{prop:Size} in the case $r=R$, followed by Lemma \ref{lem:RHn}, Remark \ref{rmk:Critical} and Theorem \ref{thm:MaybSharp} (as in Case 1.2 but relying instead on $R\leq \min\{\abs{x - y},\rho(x)\}$) shows that
\[ 
 \abs{\nabla_{1} \mathcal{E}(x,y)} 
 \lesssim_R \frac{1}{R}\|\mathcal{E}(\cdot,y)\|_{L^\infty(B(x,R/2))} \br{1 +
  \frac{V(B(x,R))}{R^{d - 2}}} 
 \lesssim_R \frac{e^{- \varepsilon d_{V}(x,y)}}{\abs{x - y}^{d - 2}},
\]
 which concludes the proof of the second estimate stated in the theorem.
 \end{prf}

Our proof of Theorem \ref{thm:Main} will also require H\"older regularity for the kernel $\nabla_{1} \mathcal{E}$. This will be deduced from the following proposition. The proof again relies on ideas from Shen's work in \cite{shen1999fundamental}.
 
\begin{prop} 
 \label{prop:Reg} 
Let $R> 0$ and $\gamma = \min \br{\alpha, \beta - 1}$. If $y\in\R^d$ and $x$, $x' \in B(y,R)$ satisfy $\abs{x - x'} \leq \frac{1}{2} \abs{x - y}$, then
 \begin{multline*}
\abs{\nabla_{1} \mathcal{E}(x,y) - \nabla_{1} \mathcal{E} (x',y)}
+ \abs{\nabla_{1} \mathcal{E}(y, x) - \nabla_{1} \mathcal{E}(y,x')} \\
\lesssim_{R}
\frac{\abs{x - x'}^{\gamma}}{\abs{x - y}^{\gamma + 1}}
\|\mathcal{E}(\cdot,y)\|_{L^\infty(B(x,\frac{3}{4}\abs{x - y}))} \br{1 +
  \frac{V(B(x,\abs{x- y}))}{\abs{x - y}^{d - 2}}},
\end{multline*}
where the implicit constant may depend only on $d$, $\lambda$, $\Lambda$, $\tau$, $\llbracket V\rrbracket_{d/2}$ and $R$.
  
\end{prop}

\begin{prf}  
Let $R>0$ and fix $y,\, x, \, x' \in \R^{d}$ as stated in the proposition. Next, setting $r := \abs{x - y}$ and $r' := \abs{x - x'}$, define the solutions $u$ and $v$ exactly as in \eqref{defu} and \eqref{defv} on the ball $B := B(x, 3r/4)$.

We then have $-\mathrm{div}A\nabla u + V u = 0$ and $-\mathrm{div}A\nabla v = 0$ in $B$. The H\"older regularity estimate \eqref{eqtn:Interior1} in Theorem \ref{thm:Interior} followed by Lemma \ref{lem:PotentialFreeSize0} then implies that
\begin{align*}\begin{split}
 \abs{\nabla v(x) - \nabla v(x')} 
 &\lesssim_R \frac{(r')^{\alpha}}{r^{1 + \alpha}} \|v\|_{L^\infty(B)} 
 \lesssim \frac{(r')^{\gamma}}{r^{1 + \gamma}} \|u\|_{L^\infty(B)} \br{\!1 + \sup_{\xi \in B} \int_{B} \frac{V(\zeta)}{\abs{\xi - \zeta}^{d - 2}} \, d\zeta\!}\!,
\end{split}\end{align*}
where the first estimate relies on the fact that $r\leq R$. If $\xi \in B(x,3r/4)$, then since $B(x,3r/4) \subset B(\xi,2r)\subset B(x,3r)$, it follows that 
\begin{align*}\begin{split}  
 \int_{B} \frac{V(\zeta)}{\abs{\xi - \zeta}^{d - 2}} \, d\zeta \leq \int_{B(\xi,2
  r)} \frac{V(\zeta)}{\abs{\xi - \zeta}^{d -2}} \, d \zeta
 \lesssim \frac{V(B(x,r))}{r^{d - 2}},
\end{split}\end{align*}
where the second estimate uses \eqref{eqtn:0.4} from Lemma \ref{lem:RHn2}, which requires $V\in RH_{\frac{d}{2}}$, and the doubling property \eqref{eqtn:0.3}. Altogether, this shows that
\begin{equation}
  \label{eqtn:Reg3}
  \abs{\nabla v(x) - \nabla v(x')} \lesssim \frac{(r')^{\gamma}}{r^{1
      + \gamma}} \|\mathcal{E}(\cdot,y)\|_{L^\infty(B)} \br{1 + \frac{V(B(x,r))}{r^{d - 2}}},
\end{equation}
as required.

We now define $w(\xi) := \int_{B} \mathcal{E}_{A}^{0}(\xi,\zeta) u(\zeta) V(\zeta) \, d\zeta$ for all $\xi \in B$ in order to estimate
\begin{align*}\begin{split}  
 \abs{\nabla w(x) - \nabla w(x')} &\leq \|u\|_{L^\infty(B)}
 \int_{B} \abs{\nabla_{1} \mathcal{E}_{A}^{0}(x,\zeta) - \nabla_{1}
   \mathcal{E}_{A}^{0}(x',\zeta)} V(\zeta) \, d\zeta \\
 &=:  \|\mathcal{E}(\cdot,y)\|_{L^\infty(B)} \br{I_{1} + I_{2} + I_{3}},
\end{split}\end{align*}
where $I_{1}$, $I_{2}$ and $I_{3}$ are defined by restricting the integral to the sets $B \cap B(x,2 r')$, $B \cap B(x',2r')$ and $\lb \zeta \in B : \abs{\zeta - x} \geq 2 r', \, \abs{\zeta - x'} \geq 2 r' \rb$, respectively.

\vspace*{0.1in}

To estimate $I_{1}$, we use the first estimate in Lemma \ref{lem:PotentialFreeSize}  to obtain
\begin{align*}\begin{split}  
 I_{1} &\lesssim_R \int_{B(x, 2r')} \frac{V(\zeta)}{\abs{\zeta - x}^{d
   - 1}} \, d\zeta + \int_{B \cap B(x,2r')} \frac{V(\zeta)}{\abs{\zeta -
   x'}^{d - 1}} \, d\zeta \\
&\lesssim \int_{B(x,2r')} \frac{V(\zeta)}{\abs{\zeta - x}^{d - 1}} \, d\zeta +
\int_{B(x',3 r')} \frac{V(\zeta)}{\abs{\zeta - x'}^{d - 1}} \, d\zeta \\
&\lesssim \frac{V(B(x,2 r'))}{\br{r'}^{d - 1}} + \frac{V(B(x',
  3 r'))}{\br{r'}^{d - 1}},
\end{split}\end{align*}
where we used \eqref{eqtn:RHn2} from Lemma \ref{lem:RHn2}, which requires $V\in RH_{d}$, in the third line. Noting that
$B(x',3 r') \subset B(x, 4 r')$ and using the volume comparison estimate \eqref{eqtn:beta} followed by the doubling property \eqref{eqtn:0.3}, we obtain
\begin{align}\begin{split}\label{eqtn:Reg4}  
 I_{1}  &\lesssim_R \br{\frac{r'}{r}}^{d - 2 + \beta} \frac{1}{\br{r'}^{d -
     1}} \br{V(B(x,2r)) + V(B(x,4r))} \\
 &\lesssim \br{\frac{r'}{r}}^{\beta - 1} \frac{1}{r}
 \frac{V(B(x,r))}{r^{d - 2}} \\
 &\leq \br{\frac{r'}{r}}^{\gamma} \frac{V(B(x,r))}{r^{d - 1}}
\end{split}\end{align}
where the last line uses that $\gamma < \beta -1$.

\vspace*{0.1in}

To estimate $I_2$, we use similar reasoning to that above for $I_1$ to obtain
\begin{equation}\label{I2}
I_2 \lesssim_R \br{\frac{r'}{r}}^{\gamma} \frac{V(B(x,r))}{r^{d - 1}}.
\end{equation}

\vspace*{0.1in}

To estimate $I_{3}$, we first use the H\"older regularity \eqref{eqtn:Interior1} in Theorem \ref{thm:Interior}, for the solution $\mathcal{E}_{A}^{0}(\cdot,\zeta)$ in the ball $B(x,\frac{3}{4}\abs{x - \zeta})$, followed by Lemma \ref{lem:PotentialFreeSize0} to obtain
\begin{align*}\begin{split}  
 \abs{\nabla_{1} \mathcal{E}_{A}^{0}(x,\zeta) - \nabla_{1} \mathcal{E}_{A}^{0}(x',\zeta)}
 &\lesssim_R \frac{\br{r'}^{\alpha}}{\abs{x - \zeta}^{1 + \alpha}} 
   \|\mathcal{E}_{A}^{0}(\cdot,\zeta)\|_{L^\infty(B(x,\frac{3}{4}\abs{x - \zeta}))}  \\
 &\lesssim \frac{\br{r'}^{\alpha}}{\abs{x - \zeta}^{1 + \alpha}} \sup_{\xi \in
 B\br{x,\frac{3}{4}\abs{x - \zeta}}} \frac{1}{\abs{\xi - \zeta}^{d - 2}} \\
 &\lesssim \frac{\br{r'}^{\alpha}}{\abs{x - \zeta}^{d - 1 + \alpha}}
 \end{split}\end{align*}
for all $\zeta \in B$, since $\abs{x - \zeta} \leq \abs{\xi - \zeta} + \frac{3}{4}\abs{x - \zeta}$ for all $\xi \in B\br{x,\frac{3}{4}\abs{x - \zeta}}$.

Using the above estimate, we obtain
\begin{align*}\begin{split}  
 I_{3} &= \int_{\{\zeta \in B : \abs{\zeta - x} \geq 2r', \, \abs{\zeta - x'} \geq 2 r'\}}  \abs{\nabla_{1} \mathcal{E}_{A}^{0}(x,\zeta) - \nabla_{1}
   \mathcal{E}_{A}^{0}(x',\zeta)} V(\zeta) \, d\zeta \\
 &\lesssim_R (r')^{\alpha} \int_{2r' \leq \abs{\zeta - x} \leq  r}
 \frac{V(\zeta)}{\abs{\zeta - x}^{d - 1 + \alpha}} \, d\zeta \\
 &\lesssim (r')^{\alpha} \int^{2r}_{r'} \frac{V(B(x,t))}{t^{d - 1 + \alpha}} \,
 \frac{dt}{t},
\end{split}\end{align*}
where the final estimate can be obtained by breaking the integral up into a finite sum over the dyadic annuli $B(x,2^{j+1}r') \setminus B(x,2^jr')$ for all $j\in\{0,\ldots,\lceil\log_2(r/r')\rceil\}$. The volume comparison estimate \eqref{eqtn:beta} then implies that
\begin{align}\begin{split}\label{betaalpha}
 I_{3} &\lesssim_R (r')^{\alpha} \br{\int^{2 r}_{r'} \br{\frac{t}{r}}^{d - 2 + \beta}
 \frac{1}{t^{d - 1 + \alpha}} \, \frac{dt}{t}} V(B(x,2r)) \\
 &\lesssim \frac{\br{r'}^{\alpha}}{r^{\beta - 1}} \br{\int^{2 r}_{r'} t^{(\beta -1) - \alpha} \, \frac{dt}{t}} \frac{V(B(x,r))}{r^{d - 1}}.
\end{split}\end{align}

To evaluate the remaining integral, consider three cases:\\
If $\beta-1>\alpha$, then
\begin{align*}\begin{split}  
 I_{3} &\lesssim_R \frac{(r')^{\alpha}}{r^{\beta - 1}} r^{(\beta-1)-\alpha}
 \frac{V(B(x,r))}{r^{d - 1}} 
  = \frac{(r')^{\alpha}}{r^{\alpha}}
 \frac{V(B(x,r))}{r^{d - 1}}.
 \end{split}\end{align*}
If $\beta-1<\alpha$, then
\begin{align*}\begin{split}  
 I_{3} &\lesssim_R \frac{(r')^{\alpha}}{r^{\beta - 1}} (r')^{(\beta-1)-\alpha}
 \frac{V(B(x,r))}{r^{d - 1}} 
  = \frac{(r')^{\beta-1}}{r^{\beta-1}}
 \frac{V(B(x,r))}{r^{d - 1}}.
 \end{split}\end{align*}
If $\beta-1=\alpha$, then recall that $\beta := 2 - \frac{d}{d+\delta}$ where $V \in RH_{d+\delta}$ for some $\delta>0$. The self-improvement property in Proposition~\ref{prop:SelfImprovement} thus implies that there exists $\delta'>\delta$ such that $V \in RH_{d+\delta'}$ whilst $\beta':=2 - \frac{d}{d+\delta'}$ satisfies $\beta'-1>\alpha$. Therefore, applying the volume comparison estimate as in \eqref{betaalpha} but with $\beta'$ instead of $\beta$, we obtain 
\begin{align*}\begin{split}  
 I_{3} &\lesssim_R \frac{(r')^{\alpha}}{r^{\alpha}}
 \frac{V(B(x,r))}{r^{d - 1}}.
 \end{split}\end{align*}
 
Altogether, the three cases above prove that
\begin{equation}
  \label{eqtn:Reg5}
  I_{3} \lesssim_R \frac{(r')^{\gamma}}{r^{1 + \gamma}} \frac{V(B)}{r^{d - 2}}.
\end{equation}
Altogether, estimates \eqref{eqtn:Reg4}, \eqref{I2} and \eqref{eqtn:Reg5} for $I_{1}$, $I_{2}$ and $I_{3}$ prove that
$$
\abs{\nabla w(x) - \nabla w(x')}
\lesssim_R \frac{(r')^{\gamma}}{r^{1 + \gamma}}  \|\mathcal{E}(\cdot,y)\|_{L^\infty(B)} \frac{V(B(x,r))}{r^{d
  - 2}},
$$
which combined with \eqref{eqtn:Reg3} proves the required estimate for $ \abs{\nabla_{1} \mathcal{E}(x,y) - \nabla_{1} \mathcal{E} (x',y)}$.

\vspace*{0.1in}

Reflecting on the proof above, we see that essentially the exact same
argument can be used to estimate the remaining term
$\abs{\nabla_{1}\mathcal{E}(y,x) - \nabla_{1}\mathcal{E}(y,x')}$. The primary difference will be to replace the solutions $u$ and $v$ from \eqref{defu} and \eqref{defv} with
$$
\tilde{u}(\xi) := \mathcal{E}(y,\xi)
\quad\textrm{ and}\quad
\tilde{v}(\xi) := \tilde{u}(\xi) + \int_{B} \mathcal{E}_{A}^{0}(\zeta,\xi) \tilde{u}(\zeta) V(\zeta) \, d\zeta
$$
for all $\xi \in B := B(x, 3r/4)$, which instead satisfy $-\mathrm{div}A^*\nabla \tilde{u} + V \tilde{u} = 0$ and $-\mathrm{div}A^*\nabla \tilde{v} = 0$ in $B$. This does not  alter the remainder of the proof, however, since the hypotheses on $A$ are preserved by its transpose $A^*$.
\end{prf}

We now deduce the aforementioned H\"older regularity for the kernel $ \nabla_{1} \mathcal{E}$.

\begin{cor}
  \label{cor:Reg}
Let $R> 0$ and $\gamma = \min \br{\alpha, \beta - 1}$. If $x$, $y\in\R^d$ and $\abs{x - y} \leq
\rho(x)$ whilst also $x, x' \in B(y,R)$ and $\abs{x - x'} \leq \frac{1}{2} \abs{x - y}$, then
\[
\abs{\nabla_1 \mathcal{E}(x,y) - \nabla_1 \mathcal{E}(x',y)}
+\abs{\nabla_1 \mathcal{E}(y,x) - \nabla_1 \mathcal{E}(y,x')} 
\lesssim_{R} \br{\frac{\abs{x - x'}}{\abs{x - y}}}^{\gamma}
\frac{1}{\abs{x - y}^{d - 1}},
\]
where the implicit constant may depend only on $d$, $\lambda$, $\Lambda$, $\tau$, $\llbracket V\rrbracket_{d/2}$ and $R$.
\end{cor}

\begin{prf}  
 If $\abs{x - y} \leq \rho(x)$, then we can use \eqref{eqtn:beta} and
 Remark \ref{rmk:Critical} to obtain
 $$
\frac{V(B(x,\abs{x - y}))}{\abs{x - y}^{d - 2}} \lesssim
\frac{V(B(x,\rho(x)))}{\rho(x)^{d - 2}} \lesssim 1.
$$
Therefore, by Proposition \ref{prop:Reg} and Theorem
\ref{thm:MaybSharp}, we obtain
$$
\abs{\nabla_1 \mathcal{E}(x,y) - \nabla_1 \mathcal{E}(x',y)}
\lesssim_R \frac{\abs{x - x'}^{\gamma}}{\abs{x - y}^{\gamma+1}}
\sup_{\xi \in B(x,\frac{3}{4} \abs{x - y})} \frac{1}{\abs{\xi - y}^{d - 2}}
\lesssim \frac{\abs{x - x'}^{\gamma}}{\abs{x - y}^{\gamma+d-1}},
$$
as required.
 \end{prf}

The following proposition will allow us to approximate the potential dependent kernel $\nabla_{1} \mathcal{E}=\nabla_{1} \mathcal{E}_A^V$ with the potential-free kernel $\nabla_{1} \mathcal{E}_A^0$ at local scales determined by the critical radius function. The proof once again relies on ideas from Shen's work in \cite{shen1999fundamental}.

\begin{prop}
 \label{prop:Flatness} 
    Let $\eta, \, M > 0$. If $x,\, y \in B(0,M)$ and $|x-y| \leq \eta \rho (x)$, then
   \begin{align*}
\abs{\nabla_{1} \mathcal{E}(x,y)  - \nabla_{1}
  \mathcal{E}_{A}^{0}(x,y)} 
  &\lesssim_{\eta,M} \rho(x)^{-\beta}  \abs{x - y}^{\beta + 1 - d}
    \leq C_\rho(M) \abs{x - y}^{\beta + 1 - d},
\end{align*}
where the implicit constant may depend only on $d$, $\lambda$, $\Lambda$, $\tau$, $\llbracket V\rrbracket_{d/2}$, $\eta$ and $M$, whilst
$$
C_\rho(M) := B_{0}^\beta\rho(0)^{-\beta} \br{1 + \frac{M}{\rho(0)}}^{M_{0} \beta} \in [0,\infty)
$$
with the constants $B_{0}$ and $M_{0}$ from Lemma \ref{lem:Shen0}.
 \end{prop}

 \begin{prf}
   Let $\eta>0$ and fix $x,\, y \in B(0,M)$ with $|x-y| \leq \eta \rho (x)$. We set  $r= \abs{x - y}$ and $R = \rho(x)$. Proposition $7.10$ of \cite{mayboroda2019exponential} allows us to express the difference of the two fundamental solutions as
   $$
\mathcal{E}_{A}^{0}(x,y) - \mathcal{E}(x,y) = \int_{\R^{d}}
\overline{\mathcal{E}_{A^{*}}^{0}(z,x)} \mathcal{E}(z,y) V(z) \, dz = \int_{\R^{d}}
\mathcal{E}_{A}^{0}(x,z) \mathcal{E}(z,y) V(z) \, dz,
$$
where the second identity uses that $A$ is real-valued with $\mathcal{E}_{A^*}^0(z,x)=\mathcal{E}_{A}^{0}(x,z)$. Therefore, we have
\begin{multline*}
 \abs{\nabla_{1} \mathcal{E}(x,y) - \nabla_{1} \mathcal{E}_{A}^{0}(x,y)}
  \leq \int_{B(x,1)} \abs{\nabla_{1} \mathcal{E}_{A}^{0}(x,z)} \mathcal{E}(z,y)
 V(z) \, dz \\ 
 + \int_{\R^{d} \setminus B(x,1)} \abs{\nabla_{1}
   \mathcal{E}_{A}^{0}(x,z)} \mathcal{E}(z,y) V(z) \, dz =: I_1 + I_2.
\end{multline*}

To estimate $I_1$, the first estimate in Lemma~\ref{lem:PotentialFreeSize} and Theorem \ref{thm:MaybSharp} show that
\begin{align*}
  I_1 &\lesssim \int_{B(x,1)} \frac{1}{\abs{x - z}^{d - 1}}
 \frac{e^{- \varepsilon d_{V}(z,y)}}{\abs{z - y}^{d - 2}} V(z) \, dz \\
 &\lesssim \frac{1}{r^{d - 2}} \int_{B(x,r)} \frac{V(z)}{\abs{x
     - z}^{d - 1}} \, dz + \frac{1}{r^{d - 1}} \br{\frac{r}{\rho(y)}}^{\beta},
\end{align*}
where the second estimate follows by the same arguments used to prove (7.15) in Lemma 7.13 of \cite{shen1999fundamental}. Consecutive applications of \eqref{eqtn:RHn2} in Lemma \ref{lem:RHn2}, which requires $V \in RH_{d}$, the volume comparison estimate \eqref{eqtn:beta} and Remark~\ref{rmk:Critical} show that
\begin{align*}\begin{split}  
 \frac{1}{r^{d - 2}} \int_{B(x,r)} \frac{V(z)}{\abs{x - z}^{d - 1}} \,
 dz &\lesssim \frac{1}{r^{d - 2}} \frac{V(B(x,r))}{r^{d - 1}} \\
 &\lesssim  \frac{1}{r^{d - 2}} \br{\frac{r}{R}}^{d - 2 + \beta}
 \frac{V(B(x,R))}{r^{d - 1}} \\
 &\lesssim \frac{1}{r^{d - 1}} \br{\frac{r}{R}}^{\beta}.
\end{split}\end{align*}
Hence, utilising Corollary \ref{cor:LocalCritical}, we obtain
\begin{align*}\begin{split}  
   I_1 \lesssim \frac{1}{r^{d - 1}} \br{\frac{r}{R}}^{\beta} +
  \frac{1}{r^{d - 1}} \br{\frac{r}{\rho(y)}}^{\beta}
 \lesssim_\eta \frac{1}{r^{d - 1}} \br{\frac{r}{R}}^{\beta}.
\end{split}\end{align*}

To estimate $I_2$, the second estimate in Lemma~\ref{lem:PotentialFreeSize} and Theorem~\ref{thm:MaybSharp} show that
$$
I_2 \lesssim \int_{\R^{d}} \frac{1}{\abs{z - x}^{d - 2}} \frac{e^{-\varepsilon d_{V}(z,y)}}{\abs{z - y}^{d - 2}} V(z) \, dz
\lesssim \frac{1}{r^{d - 2}} \br{\frac{r}{\rho(y)}}^{\beta}
 \lesssim_{\eta} \frac{M}{r^{d - 1}} \br{\frac{r}{R}}^{\beta},
$$
where the second estimate follows from the arguments beginning at (4.10) in the proof of Lemma $4.8$ of \cite{shen1999fundamental}, and the third estimate utilises Corollary \ref{cor:LocalCritical} and the fact that $r=|x-y| \leq 2M$. 

Altogether, we have shown that
$$
I_1+I_2 \lesssim_{\eta} \frac{(1+M)}{r^{d - 1}} \br{\frac{r}{R}}^{\beta} \lesssim_{\eta,M} \frac{1}{r^{d - 1}} \br{\frac{r}{R}}^{\beta}.
$$
The proof of this proposition is then completed by combining this
estimate with Corollary \ref{cor:Shen0}.
\end{prf}

For a constant coefficient matrix $A_0$ satisfying the ellipticity in \eqref{eqtn:Ellipticity}, we will use $\Theta(x,y;A_0) := \mathcal{E}_{A_0}^0(x,y)$ to denote the fundamental solution for the constant coefficient operator $-\mathrm{div} A_0 \nabla$. Our proof of Theorem \ref{thm:Main} will rely on the antisymmetry of the derivative kernel whereby $\nabla_{1} \Theta(x,y;A_0) = - \nabla_{1} \Theta (y,x;A_0)$ for all $x, \, y \in \R^{d}$. The following estimate will allow us to exploit this antisymmetry by showing how the kernel $\nabla_{1}
\mathcal{E}_A^V(x,y)$ for a variable coefficient $A$ can be approximated, at local scales determined by the critical radius function, by the kernels for constant coefficient operators obtained by so-called freezing of the coefficients of $A$.

\begin{cor}
  \label{cor:Flatness}
   Let $\eta, \, M > 0$. If $x, \, y \in B(0,M)$ and $\abs{x - y} \leq
   \eta \rho (x)$, then
$$
\abs{\nabla_1 \mathcal{E}(x,y)  - \nabla_1 {\Theta}(x,y;{A(x)})} 
+ \abs{\nabla_1 \mathcal{E}(x,y)  - \nabla_1 {\Theta}(x,y;{A(y)})} 
\lesssim_{\eta,M} \abs{x - y}^{\beta + 1 - d},
$$
where the implicit constant may depend only on $d$, $\lambda$, $\Lambda$,
$\tau$, $\llbracket V \rrbracket_{d/2}$, $\eta$ and $M$.
\end{cor}

\begin{prf}  
 This follows immediately from Proposition~\ref{prop:Flatness} in combination with estimates (b) and (c) from Lemma 2.2 in \cite{conde2019failure}.
 \end{prf}

 \section{Proof of Theorem \ref{thm:Main}}
 \label{sec:Proof}

The proof will use a variational argument and maximum principle that
are now very standard in this area. The ones who introduced such an argument in this context were Eiderman, Nazarov and Volberg \cite{eiderman2014s}. Our proof will be closer to the potential-free one from \cite{conde2019failure}. We will pay special attention to those parts of the argument that differ substantially from the potential-free case and refer the reader to \cite{conde2019failure} if the
proof of a statement in our potential dependent setting is the same as
in \cite{conde2019failure}, rather than repeating the argument
verbatim.

Throughout this section we continue to consider a coefficient matrix
$A$ on $\R^d$ for some $d\geq3$ satisfying the assumptions in Section
\ref{sect:Prelim} with constants $\lambda, \, \Lambda,\, \alpha, \,
\tau > 0$. We also assume that $V$ is a fixed non-negative locally
integrable function on $\R^d$ that belongs to the reverse H\"older
class $RH_d$. 
We continue to use the shorthand notation $L:=L_{A}^{V}$,
$\mathcal{E}:=\mathcal{E}_{A}^{V}$ and introduce
$T_\mu:=T_{A,\mu}^{V}$. Moreover, given a Borel measure $\sigma$ on
$\R^{d}$ define
$$
T\sigma(x) := T^V_{A,\sigma}(1)(x) = \int \nabla_{1} \mathcal{E}(x,y) \, d\sigma(y)
$$
and set $\norm{\sigma} := \sigma(\R^{d})$.

\subsection{Reduction to a Localized Estimate}\label{ss:red} It can be assumed, without loss
of generality, that there exists some $\tau_{0} > 0$ such that
$$
\Theta^{d - 1,*}(x,\mu) > \tau_{0}
$$
for $\mu$-a.e. $x \in \R^{d}$. To see this, note that $\mu$ can be restricted to a suitable subset with positive $\mu$ measure for which such a
$\tau_{0}$ does exist. The unboundedness of $T_\mu$ will then follow
from the unboundedness of the operator acting on the restricted subset.

Similarly, it can also be assumed that $\mu$ has $(d - 1)$-polynomial growth with
constant $c_{0} > 0$. This follows from the fact that $\Theta^{d-1,*}(x,\mu)$
is $\mu$-a.e. finite. Indeed, since $\Theta^{d-1,*}(x,\mu)$ is
$\mu$-a.e. finite, by restricting to a suitable subset with positive
$\mu$ measure if necessary, it can be assumed that there exists some $N
> 0$ for which $\Theta^{d-1,*}(x,\mu) \leq N$ for $\mu$-a.e. $x \in
\R^{d}$. This implies that there must exist some $\varepsilon > 0$ for which
$$
\frac{\mu(B(x,r))}{(2 r)^{d-1}} \leq  2 N
$$
for all $r \leq \varepsilon$. For $r > \varepsilon$, the estimate
$\mu(B(x,r)) \lesssim r^{d-1}$ follows from the fact that $\mu$ is
compactly supported.

\begin{definition} 
  \label{def:Densities}
  For a ball $B \subset \R^{d}$, define the $(d - 1)$-dimensional
  density of $B$ through
  $$
\Theta_{\mu}(B) := \frac{\mu(B)}{\mathrm{diam}(B)^{d - 1}}.
$$
Similarly, for a cube $Q \in \mathcal{D}$, the $(d - 1)$-dimensional
density of $Q$ is defined through
$$
\Theta_{\mu}(Q) := \frac{\mu(Q)}{l(Q)^{d - 1}}.
$$
 \end{definition}

The following lemma is purely a property of the measure and does not
depend on the differential operator under consideration. It will therefore remain true in our context.

\begin{lem}[{\cite[Lemma~4.1]{conde2019failure}}]
  \label{lem:CubeSequence}
  For $\mu$-a.e. $x \in \R^{d}$ there exists a sequence of
  high-density cubes in $\mathcal{D}^{db}$ containing $x$ whose length vanishes. More precisely,
  there exists a sequence $\lb Q_{k} \rb_{k \in \N} \in \mathcal{D}^{db}$ such that
  $l(Q_{k}) \rightarrow 0$ as $k \rightarrow \infty$, $x \in Q_{k}$
  and $\Theta_{\mu}(Q_{k}) > c  \tau_{0}$ for all $k \in \N$, where $c
  > 0$ is some constant that only depends on the dimension and the
  parameters of the David--Mattila lattice.

    For $\mu$-a.e. $x \in \R^{d}$ there exists a sequence of
  low-density cubes in $\mathcal{D}^{db}$ containing $x$ whose length vanishes. More precisely, for
  any $A > 1$ and $\delta \in (0,1)$,
  there exists a sequence $\lb Q_{k} \rb_{k \in \N} \in \mathcal{D}^{db}$ such that
  $l(Q_{k}) \rightarrow 0$ as $k \rightarrow \infty$, $x \in Q_{k}$
  and $\Theta_{\mu}(A B_{Q_{k}}) < \delta$ for all $k \in \N$.
\end{lem}

With the existence of high and low-density sequences of cubes ascertained,
the following high and low density sub-collections can be
introduced. Fix $A > 1$ and $\delta \in (0,1)$, the values of which
will be determined at a later time. Let $Q \in \mathcal{D}$,
set $\tau := c \tau_{0}$ with $c$ as given in the previous lemma and define
$$
HD(Q) := \lb R \subsetneq Q :  R \in \mathcal{D}^{db}, \,
\Theta_{\mu}(R) > \tau, \, R \ maximal \rb.
$$
Also let
$$
LD(Q) := \lb R \subsetneq Q : R \in \mathcal{D}^{db}, \,
\Theta_{\mu}(A B_{R}) \leq \delta, \, R \ maximal \rb.
$$
It is clear that both $LD(Q)$ and $HD(Q)$ partition the cube $Q$.
Set
$$
\Sigma_{0} := \lb Q_{k_{0}} \rb,
$$
where if you recall $Q_{k_{0}} := \supp  \mu$.
Then, for $Q \in \mathcal{D}$ introduce
$$
\Sigma_{1}(Q) := \sum_{R \in HD(Q)} LD(R)
$$
and inductively define
$$
\Sigma_{k + 1} := \bigcup_{Q \in \Sigma_{k}} \Sigma_{1}(Q)
$$
for each $k \geq 0$. Since $LD(Q)$ and $HD(Q)$ both partition $Q$ for
any $Q \in \mathcal{D}$, it follows that $\Sigma_{k}$ partitions
$\supp  \mu$ for any $k \in \N$. $\Sigma = \lb \Sigma_{k}
\rb_{k \in \N}$ is thus a filtration of low-density cubes. Define the
martingale difference
$$
\Delta_{Q}f = \sum_{S \in \Sigma_{1}(Q)} \langle f \rangle_{S}
\chi_{S} - \langle f \rangle_{Q} \chi_{Q}
$$
for $Q \in \Sigma$. The function $f$ may then be decomposed in the $L^{2}(\mu)$-sense as
$$
f = \langle f \rangle_{Q_{k_{0}}} + \sum_{Q \in \Sigma} \Delta_{Q}f.
$$
The orthogonality of the martingale differences then leads to
$$
\norm{T \mu}^{2}_{L^{2}(\mu)} = (\langle T \mu
\rangle_{Q_{k_{0}}})^{2} \mu(Q_{k_{0}}) + \sum_{Q \in \Sigma} \norm{\Delta_{Q}(T\mu)}^{2}_{L^{2}(\mu)}.
$$
This decomposition allows us to reduce the proof of Theorem \ref{thm:Main} to the proof of the following proposition.

\begin{prop} 
 \label{prop:InitialReduction} 
 Suppose that $T_{\mu}$ is bounded on $L^{2}(\mu)$. There must then exist
 some $N_{0} > 0$ such that if $Q \in \Sigma_{N}$ for $N > N_{0}$ and
 $\delta$  is chosen small enough then
 \begin{equation}
   \label{eqtn:InitialReduction}
\norm{\Delta_{Q}\br{T \mu}}^{2}_{L^{2}(\mu)} \gtrsim_{\tau} \mu(Q).
\end{equation}
\end{prop}

From this point on, we will assume that $T_{\mu}$ is bounded and we aim
to prove the lower estimate \eqref{eqtn:InitialReduction}.

The following result states that it is possible to increase the
generation of the cubes in the David--Mattila lattice sufficiently high
so that all of the cubes are smaller than the critical radius of the
potential. This will allow us to utilize the local kernel estimates
from Section \ref{sec:KernelEstimates} at the scale of that generation. Recall that for a cube $Q \in
\mathcal{D}_{k}$ we have $Q \subset B_{Q} := B(x_{Q}, 28 r(Q))$.

\begin{prop} 
 \label{prop:AdmissibleDM} 
There exists $K_{0} >0$ large enough so that the diameter of $B_{Q}$ is
smaller than $\mathrm{min}(1/2,\inf_{\supp  \mu} \rho)$
for all cubes $Q \in \mathcal{D}_{k}$ with $k \geq K_{0}$.
\end{prop}

\begin{prf}  
 Observe that the radius $r(Q)$ decreases by a fixed amount in
each generation. In particular, from the definition of the
David--Mattila lattice we know that
$$
A_{0}^{-k} \leq r(Q) \leq C_{0} A_{0}^{-k}
$$
for $Q \in \mathcal{D}_{k}$ and $k \in \N$, where $A_{0}$ and $C_{0}$
are the parameters of the lattice.
Let $D$ be large enough so that $\supp  \mu \subset B(0,D)$.
Fix $K_{0}$ large enough so that
$$
C_{0} A_{0}^{-K_{0}} \leq \frac{1}{56} \mathrm{min} \br{\frac{1}{2},  B_{0}^{-1} \rho(0) \br{1 +  D / \rho(0)}^{-M_{0}}},
$$
where $B_{0}$ and $M_{0}$ are the constants from Lemma
\ref{lem:Shen0}. Corollary \ref{cor:Shen0} then implies that
$$
56 r(Q) \leq \min \br{\frac{1}{2}, \rho(y)}
$$
for all $y$ contained in the support of $\mu$ and $Q \in
\mathcal{D}_{K_{0}}$. This proves that the diameter of $B_{Q}$ is
smaller than $\mathrm{min} \br{1/2,\inf_{\supp  \mu} \rho}$
for any $Q \in \mathcal{D}_{k}$ with $k \geq K_{0}$.
\end{prf}

Assume that $Q \in \mathcal{D}_{k}$ for $k \geq K_{0}$. Since $Q \subset
B_{Q}$, it will then be true that
$$
\mathrm{diam}(Q) \leq \mathrm{min} \br{\frac{1}{2},
  \inf_{\supp \mu} \rho}.
$$
For some $\varepsilon_{0} > 0$ to be chosen later on, define
$\Sigma_{1}'(Q)$ to be a finite subcollection of $\Sigma_{1}(Q)$ that
satisfies
$$
\mu \br{\bigcup_{S \in \Sigma'_{1}(Q)} S} > \br{1 - \varepsilon_{0}} \mu(Q).
$$
Given some small $\kappa_{0} \in (0,1)$, to be fixed at a later time,
and $S \in \Sigma_{1}'(Q)$, define the auxilliary region
$$
I_{\kappa_{0}}(S) := \lb x \in S : \mathrm{dist}(x, \supp  \mu
\setminus S) \geq \kappa_{0} l(S) \rb.
$$
Define the localized measure $\sigma$ through
$$
\sigma = \sigma_{Q} = \sum_{S \in \Sigma_{1}'(Q)}
\frac{\mu(I_{\kappa_{0}}(S))}{\mathcal{L}^{d}
  \br{\frac{1}{4}B(S)}} \mathcal{L}^{d} \vert_{\frac{1}{4}B(S)},
$$
where $\mathcal{L}^{d}$ denotes the Lebesgue measure on $\R^{d}$.
The small boundary condition \eqref{eqtn:SmallBoundary} taken together with the doubling
property implies
$$
\mu(S \setminus I_{\kappa_{0}}(S)) \lesssim \kappa_{0}^{\frac{1}{2}}
\mu(3.5 B_{S}) \lesssim \kappa_{0}^{\frac{1}{2}} \mu(S).
$$
Therefore,
\begin{align*}\begin{split}  
 \mu(S) &= \mu(I_{\kappa_{0}}(S)) + \mu(S \setminus I_{\kappa_{0}}(S))
 \\
 &\leq \mu(I_{\kappa_{0}}(S)) + c \kappa_{0}^{\frac{1}{2}} \mu(S),
\end{split}\end{align*}
for some $c > 0$.
For $\kappa_{0}$ chosen small enough this will give
$$
\mu(S) \leq 2 \mu(I_{\kappa_{0}}(S)).
$$
Therefore,
\begin{align*}\begin{split}  
\sigma(\R^{d}) &= \sum_{S \in \Sigma_{1}'(Q)}
\mu(I_{\kappa_{0}}(S)) 
\geq \frac{1}{2} \sum_{S \in \Sigma'_{1}(Q)} \mu(S) \\
&= \frac{1}{2} \mu \br{\cup_{S \in \Sigma'_{1}(Q)} S} 
> \frac{1}{2} (1 - \varepsilon_{0}) \mu(Q).
\end{split}\end{align*}
For $\varepsilon_{0}$ selected to be smaller than $1/3$, we then have
\begin{equation}
  \label{eqtn:SigmaMu}
\mu(Q) \leq 3 \sigma(\R^{d}).
\end{equation}
This together with the trivial estimate $\sigma(\R^{d}) \lesssim
\mu(Q)$ then implies $\sigma \br{\R^{d}} \simeq  \mu(Q)$.

\begin{lem}
  \label{lem:5.1}
  For any $\varepsilon > 0$, there exists $N_{0} > 0$ such that if $N
  > N_{0}$, $\kappa_{0}$ and $\varepsilon_{0}$ are small enough, $A$
  is big enough and $\delta$ is small enough, then there must exist $c > 0$
  such that
  $$
\norm{\Delta_{Q} T \mu}^{2}_{L^{2}(\mu)} \geq c \norm{T
  \sigma}^{2}_{L^{2}\br{\sigma}} - \varepsilon \mu(Q)
$$
for any $Q \in \Sigma_{N}$.
\end{lem}

\begin{prf}
  The proof of the potential-free analogue of this lemma,
  {\cite[Lemma~5.1]{conde2019failure}}, is heavily reliant on the regularity and
  size estimates given by {\cite[Lemma~2.1]{conde2019failure}} and
  parts (2) and (3) of the
  freezing coefficients lemma {\cite[Lemma~2.1]{conde2019failure}}. For
  our case, the kernel $\nabla_{1}\mathcal{E}$ has been proved to satisfy the same estimates
  in Corollary~\ref{cor:Reg}, Theorem~\ref{thm:Size} and Corollary~\ref{cor:Flatness} subject to an additional locality restriction
  dependent on the critical radius function. We choose $N_0$ large enough so that $Q \in \Sigma_{N}$
  for $N > N_{0}$ implies that $Q \in \mathcal{D}_{k}$ for some $k \geq
  K_{0}$, where $K_{0}$ is as given in Proposition~\ref{prop:AdmissibleDM}. Then, by Proposition~\ref{prop:AdmissibleDM}, the locality restrictions in Corollary~\ref{cor:Reg} and Corollary~\ref{cor:Flatness} are satisfied and the argument from \cite{conde2019failure} can be
  applied to our case verbatim using these results.
\end{prf}

The previous lemma reduces the task of proving Proposition
\ref{prop:InitialReduction} to the following proposition.

\begin{prop}
  \label{prop:SecondReduction}
  There exists $N_{0} > 0$ such that for $Q \in \Sigma_{N}$ with $N > N_{0}$,
  $$
\norm{T \sigma}^{2}_{L^{2} \br{\sigma}} \gtrsim_{\tau} \sigma
\br{\R^{d}}.
$$
  \end{prop}

  The next section will be dedicated to a proof of this proposition.

  \subsection{Contradiction Argument}
  \label{subsec:Contradiction}

  Similar to the potential-free case, Proposition
  \ref{prop:SecondReduction} can be proved using a variational
  argument. Such an argument begins by assuming that for $Q \in \Sigma$ and $0 <
  \lambda < 1$,
  $$
\norm{T \sigma}_{L^{2}\br{\sigma}}^{2} \leq \lambda \norm{\sigma}.
$$
It will then be shown, through contradiction, that $\lambda$ can not be
made arbitrarily small for $Q \in \Sigma_{N}$ with $N > N_{0}$ large
enough.

\vspace*{0.1in}

Define the family of functions
$$
\mathcal{A} := \lb g \in L^{\infty}\br{\sigma} : g \geq 0 \ and \ \int
g \, d \sigma = \norm{\sigma} \rb.
$$
Let $F$ be the functional on $\mathcal{A}$ defined through
$$
F(g) := \lambda \norm{g}_{L^{\infty}(\sigma)} \norm{\sigma} + \int \abs{T\br{g
    \sigma}}^{2} g \, d \sigma
$$
for $g \in \mathcal{A}$.

\begin{lem}
  \label{lem:Minimizer}
  There exists $b \in \mathcal{A}$ with
  $\norm{b}_{L^{\infty}\br{\sigma}} \leq 2$ that satisfies
  $$
F(b) = \inf_{g \in \mathcal{A}} F(g).
  $$
  \end{lem}

  \begin{prf}  
It is clear that since
$$
F(\chi_{Q}) = \lambda \norm{\sigma} + \int \abs{T \sigma}^{2} \, d
\sigma \leq 2 \lambda \norm{\sigma}
$$
the function $F$ must attain its infimum over the functions
$$\tilde{\mathcal{A}} := \lb g \in \mathcal{A} : \norm{g}_{L^{\infty}\br{\sigma}}
\leq 2 \rb.$$
Let $\lb b_{k} \rb_{k} \subset \tilde{\mathcal{A}}$ be a minimising sequence
so that $F(b_{k}) \rightarrow \inf_{g
  \in \mathcal{A}} F(g)$. The Banach-Alaoglu Theorem states that we
may pass to a subsequence that converges to some $b \in \tilde{\mathcal{A}}$ in the weak-$*$
topology.
 That is,
\begin{equation}
  \label{eqtn:Weak*}
\int b_{k} g \, d \sigma \xrightarrow{k \rightarrow \infty} \int b g \, d \sigma
\end{equation}
for all $g \in L^{1}(\sigma)$. This sequence will also satisfy
$$
  \norm{b}_{L^{\infty}\br{\sigma}} \leq \liminf_{k \rightarrow \infty} \norm{b_{k}}_{L^{\infty}\br{\sigma}}.
$$
Let's prove that
\begin{equation}
  \label{eqtn:Minimizer1}
\int \abs{T(b_{k} \sigma)}^{2} b_{k} \, d \sigma \xrightarrow{k
  \rightarrow \infty} \int
\abs{T(b \sigma)}^{2} b \, d \sigma.
\end{equation}
Since $\nabla_{1}\mathcal{E}(x, \cdot)$
is contained in $L^{1}(\sigma)$ for each $x \in \R^{d}$ it follows
from \eqref{eqtn:Weak*}
that $T(b_{k})(x) \rightarrow T(b)(x)$ as $k \rightarrow \infty$ for each $x \in \R^{d}$. Then,
\begin{multline*}
 \abs{\int \abs{T (b_{k}\sigma)}^{2} b_{k} \, d \sigma - \int \abs{T(b
     \sigma)}^{2} b \, d \sigma}\\
      \leq \int \abs{ \abs{T(b_{k}
     \sigma)}^{2} - \abs{T(b \sigma)}^{2}} b_{k} d \sigma + \int
 \abs{T(b \sigma)}^{2} (b_{k} - b) \, d \sigma
\end{multline*}
The pointwise convergence of $T(b_{k}\sigma)$ to $T(b \sigma)$
together with $\norm{b_{k}}_{L^{\infty}(\sigma)} \leq 2$ implies that the first
term must converge to zero. Similarly, \eqref{eqtn:Weak*} implies that
the second term must converge to zero.

With \eqref{eqtn:Minimizer1} now established, the proof of our lemma
can be completed. We have
\begin{align*}\begin{split}  
 \inf_{g \in \mathcal{A}} F(g) = \inf_{g \in \tilde{\mathcal{A}}}
 F(g) &= \lim_{k \rightarrow \infty} \lambda
 \norm{b_{k}}_{L^{\infty}(\sigma)} \norm{\sigma} +
 \int \abs{T (b_{k}\sigma )}^{2} \, d \sigma \\
 &\geq \lambda \liminf_{k \rightarrow \infty}
 \norm{b_{k}}_{L^{\infty}\br{\sigma}} \norm{\sigma} + \int \abs{T (b \sigma)}^{2} \,
 d \sigma \\
 &\geq \lambda \norm{b}_{L^{\infty}(\sigma)} \norm{\sigma} + \int \abs{T (b \sigma)}^{2} \, d
 \sigma \\
 &= F(b),
\end{split}\end{align*}
proving that $b$ is in fact a minimiser for $F$.
 \end{prf}

Define the measure $\nu$ through
$$
d \nu := b \, d \sigma.
$$
Given a sequence $\omega = (\omega_{1},\cdots, \omega_{d})$ of Borel
measures on $\R^{d}$, define
$$
T^{*} \omega(x) := \int \nabla_{1} \mathcal{E}(y,x) \cdot d \omega(y).
$$
We use the standard variational argument \cite{eiderman2014s} to obtain the following pointwise estimate. Since the actual nature of the operator is not used in the proof, just the boundedness and the existence of the minimizer, we refer the reader to {\cite[Section~6.1]{conde2019failure}} for the proof.

\begin{lem}
  \label{lem:Pointwise}
  For $\nu$-almost every $x
  \in \supp(\nu)$,
  $$
\abs{T \nu(x)}^{2} + 2 T^{*} \br{\brs{T \nu} \nu} \leq 6 \lambda.
  $$
\end{lem}

In order to extend this pointwise estimate to all of $\R^{d}$ we will
make use of the following maximum principle satisfied by the operator $T^{*}$.

\begin{thm}
  \label{thm:TMaxPrinc}
  For any vector-valued measure $\omega$ that is
compactly supported and absolutely continuous with respect to the
Lebesgue measure with bounded density function,
\begin{equation}
  \label{eqtn:MaximumPrinc}
\sup_{x \in \R^{d}} \abs{T^{*} \omega(x)} \leq \sup_{x \in
  \supp(\omega)} \abs{T^{*}\omega(x)}.
\end{equation}
\end{thm}

\begin{proof}
Let $d \omega =
\vec{G} \, d \mathcal{L}^{n + 1}$ for some bounded vector-valued
function $\vec{G}$ and consider a test-function $\varphi \in C^{\infty}_{c} \br{\R^{d} \setminus \supp
  \, \vec{G}}$. We have by Fubini's Theorem,
\begin{align*}\begin{split}  
 \int &A^{*}(x) \nabla_{x} T^{*}\omega(x) \cdot \nabla \varphi(x) +
 V(x) T^{*} \omega(x) \varphi(x) \, dx \\ & \quad = \int A^{*}(x) \nabla_{x}
 \int \nabla_{y} \mathcal{E}(y,x) \cdot \vec{G}(y) \, dy \cdot \nabla
 \varphi(x) \, dx \\ & \qquad \qquad \qquad \qquad + \int V(x) \int \nabla_{y} \mathcal{E}(y,x) \cdot
 \vec{G}(y) \, dy \, \varphi(x) \, dx \\
 & \quad = \int \nabla_{y} \br{\int A^{*}(x) \nabla_{x} \mathcal{E}(y,x) \cdot
 \nabla \varphi(x) + V(x) \mathcal{E}(y,x) \varphi(x) \, dx} \cdot
 \vec{G}(y) \, dy.
\end{split}\end{align*}
Theorem $3.6$ of \cite{davey2018fundamental} states that
$$
\mathcal{E}(y,x) = \mathcal{E}_{A}^{V}(y,x) = \mathcal{E}_{A^{*}}^{V}(x,y)
$$
for all $x, \, y\in \R^{d}$. Therefore, from this and \eqref{eqtn:Fundamental}, we obtain
\begin{align*}\begin{split}  
  \int &A^{*}(x) \nabla_{x} T^{*}\omega(x) \cdot \nabla \varphi(x) +
 V(x) T^{*} \omega(x) \varphi(x) \, dx \\ & \quad = \int \nabla_{y} \int A^{*}(x) \nabla_{x} \mathcal{E}_{A^{*}}^{V}(x,y) \cdot
 \nabla \varphi(x)  + V(x) \mathcal{E}_{A^{*}}^{V}(x,y) \varphi(x) \, dx \cdot
 \vec{G}(y) \, dy \\
 & \quad = \int \nabla \varphi(y) \cdot \vec{G}(y) \, dy 
= 0.
\end{split}\end{align*}
This proves that $T^{*} \omega$ is $L_{A^{*}}^{V}$-harmonic on $\R^{d} \setminus\supp(\omega)$.  Moreover, the definition of $T^{*}$ and the H\"older continuity in Corollary~\ref{cor:Reg} guarantee that $x \mapsto T^{*}\omega(x)$ is a continuous function on $\R^{d}$. Therefore, the weak maximum principle \eqref{thm:MaxPrinc} for the operator $L_{A^{*}}^{V}$ on the open set $\R^{d} \setminus\supp(\omega)$ implies that
$$
\sup_{\R^{d} \setminus \supp(\omega)}
\abs{T^{*} \omega} = \sup_{\partial (\supp(\omega))} \abs{T^{*}\omega},
$$
hence
$$
 \sup_{\R^{d}} \abs{T^{*} \omega} = \mathrm{max}
 \left\{\sup_{\supp(\omega)}
   \abs{T^{*}\omega}, \sup_{\partial(\supp(\omega))} \abs{T^{*}\omega}\right\} 
 \leq \sup_{\supp (\omega)} \abs{T^{*}\omega},$$
as required.
\end{proof}

At this stage, the maximum principle for $T^{*}$ can be combined with Lemma \ref{lem:Pointwise} to obtain the  global pointwise estimate below.

\begin{lem}
  \label{lem:GlobalPointwise}
  For all $x \in \R^{d}$,
  $$
\abs{T \nu(x)}^{2} + 4 T^{*} \br{\brs{T \nu}\nu}(x) \lesssim \lambda + l(Q)^{\gamma}.
  $$
\end{lem}

\begin{prf}  
Using the estimate \eqref{eqtn:MaximumPrinc} and Theorem \ref{thm:TMaxPrinc}, the proof proceeds verbatim to
\cite{conde2019failure} with Theorem \ref{thm:Size} and Corollary \ref{cor:Flatness} replacing the use of
{\cite[Lemma~2.1(c)]{conde2019failure}} and parts (2) and (3) of
{\cite[Lemma~2.2]{conde2019failure}}. However, in order to use
Corollary \ref{cor:Flatness} the size of the cube must be small enough
so that $\abs{\xi - \zeta} \leq \inf_{z \in Q} \rho(z)$ for all $\xi, \, \zeta \in
Q$. Once again, this is ensured by Proposition \ref{prop:AdmissibleDM} by choosing $N_0$ large
enough so that $Q \in \Sigma_{N}$ for $N > N_{0}$ implies that $Q \in
\mathcal{D}_{k}$ for some $k \geq K_{0}$ with $K_{0}$ as given in
Proposition \ref{prop:AdmissibleDM}. 
\end{prf}

The rest of the proof differs substantially from the potential-free case. We have defined a Riesz transform that is sufficiently smooth but we pay the price of not having a reproducing formula for it. The way around this difficulty will be described in detail below.

  For $R \in HD(Q)$, let $\varphi_{R}$ be a smooth function that
  satisfies
  $$
\chi_{1.5 B_{R}} \leq \varphi_{R} \leq \chi_{2 B_{R}} \quad and \quad
\norm{\nabla \varphi_{R}}_{L^{\infty}(\mathcal{L}^{d})} \lesssim l(R)^{-1}.
$$
Then define the vector function $g_{R,\nabla}$ and the scalar function
$g_{R,V}$ through
  $$
g_{R,\nabla} := A^{*} \nabla \varphi_{R} \quad and \quad g_{R,V} := V^{\frac{1}{2}}
\varphi_{R}.
$$
It is clear that
$$
\supp  g_{R,\nabla} \subset 2 B_{R}, \quad
\norm{g_{R,\nabla}}_{L^{\infty}\br{\mathcal{L}^{d}}} \lesssim
l(R)^{-1}, \quad \norm{g_{R,\nabla}}_{L^{1}\br{\mathcal{L}^{d}}}
\lesssim l(R)^{d - 1} \simeq \mu(R).
$$
For $g_{R,V}$, $\supp  g_{R,V} \subset 2 B_{R}$ is also clear and we have the
following estimate on its $L^{1}(\mathcal{L}^{d})$-norm.

\begin{lem}
  \label{lem:L1gR}
    For all $R \in HD(Q)$ we have
    \begin{equation}
      \label{eqtn:L1gR}
\norm{g_{R,V}}_{L^{1}\br{\mathcal{L}^{d}}} \lesssim
\norm{V}_{L^{d}\br{2 B_{R}}}^{\frac{1}{2}} \cdot l(R)^{d - 1} \simeq \norm{V}_{L^{d}\br{2 B_{R}}}^{\frac{1}{2}} \cdot \mu(R).
    \end{equation}
  \end{lem}

  \begin{prf}  
    On expanding the $L^{1}$-norm and using H\"{o}lder's inequality,
    \begin{align*}\begin{split}  
 \norm{g_{R,V}}_{L^{1}\br{\mathcal{L}^{d}}} &= \int_{\R^{d}}
 V^{\frac{1}{2}} \varphi_{R} \, d \mathcal{L}^{d} \\
 &\lesssim \br{\int_{2 B_{R}} V^{d}}^{\frac{1}{2 d}} l(R)^{\frac{2d
     - 1}{2}}.
\end{split}\end{align*}
As $\frac{2d - 1}{2} \geq d - 1$ and the cube $R$ has length less than one we 
immediately obtain \eqref{eqtn:L1gR}. \phantom{space}
\end{prf}

In the above lemma, it is useful to remember that $RH_{d} \subset
L^{d}_{\loc}$ and therefore the quantities
$\norm{V}_{L^{d}\br{2B_{R}}}$ are finite.

Define the operator $S$ through
$$
S \omega(x) := S_{A}^{V} \omega(x) := \int_{\R^{d}} V^{\frac{1}{2}}(y) \mathcal{E}(x,y) \, d \omega(y)
$$
for measures $\omega$ on $\R^{d}$. Also define
$$
S^{*} \omega := \int_{\R^{d}} V^{\frac{1}{2}}(y) \mathcal{E}(y,x)
\, d \omega(y) = \int_{\R^{d}} V^{\frac{1}{2}}(y)
\mathcal{E}_{A^{*}}^{V}(x,y) \, d \omega(y) = S_{A^{*}}^{V}\omega(x).
$$
Then, we have the following reproducing formula by \eqref{eqtn:Fundamental},
\begin{align}\begin{split}
    \label{eqtn:Reproducing}
 \varphi_{R}(x) &= \int \nabla_{1} \mathcal{E}(y,x) \cdot A^{*}(y) \nabla
 \varphi_{R}(y) \, dy + \int V^{\frac{1}{2}}(y) \mathcal{E}(y,x)
 V^{\frac{1}{2}}(y) \varphi_{R}(y) \, dy \\
 &= T^{*} \br{A^{*} \nabla \varphi_{R} \, d  \mathcal{L}^{d}}(x) + S^{*}
 \br{V^{\frac{1}{2}} \varphi_{R} \, d \mathcal{L}^{d}}(x) \\
 &= T^{*} \br{g_{R, \nabla} \, d \mathcal{L}^{d}}(x) + S^{*}
 \br{g_{R,V} \, d \mathcal{L}^{d}}(x).
 \end{split}\end{align}

Note that our version of the reproducing formula doesn't only involve the Riesz transform $T$ but a second operator $S$ defined in terms of the potential. We will prove that the contribution from this second term is small as long as the cubes considered are small in Lebesgue measure.

We proceed with the proof and define the subcollection of cubes
$$
HD_{0}(Q) := \lb R \in HD(Q) : \nu(1.5 B_{R}) \geq \frac{1}{4} \mu(R) \rb.
$$
The norm of $\nu$ can be estimated from above by
\begin{align*}\begin{split}  
 \norm{\nu} &\leq \sum_{R \in HD_{0}(Q)} \nu(1.5 B_{R}) + \sum_{R \in
   HD(Q) \setminus HD_{0}(Q)} \nu(1.5 B_{R}) \\
 &\leq \sum_{R \in HD_{0}(Q)} \nu(1.5 B_{R}) + \frac{1}{4} \sum_{R \in HD(Q)
   \setminus HD_{0}(Q)} \mu(R) \\
 &\leq \sum_{R \in HD_{0}(Q)} \nu(1.5 B_{R}) + \frac{1}{4} \mu(Q) \\
 &\leq \sum_{R \in HD_{0}(Q)} \nu(1.5 B_{R}) + \frac{3}{4} \norm{\nu},
\end{split}\end{align*}
where the last line follows from \eqref{eqtn:SigmaMu}.
This implies that
\begin{equation}
  \label{eqtn:nunu}
  \norm{\nu} \leq 4 \sum_{R \in HD_{0}(Q)} \nu(1.5 B_{R}).
\end{equation}
As stated in \cite{conde2019failure}, $\nu$ will be doubling on the
balls $B_{R}$ for $R \in HD_{0}(Q)$ in the sense that
$$
\nu (9 B_{R}) \lesssim_{\tau} \nu(1.5 B_{R}).
$$
A Vitali type covering argument will then produce a finite subfamily
$HD_{1}(Q) \subset HD_{0}(Q)$ such that the balls $3 B_{R}$ are
pairwise disjoint for $R \in HD_{1}(Q)$ and also
\begin{equation}
  \label{eqtn:HD1}
\mu(Q) \simeq_{\tau} \norm{\nu} \simeq_{\tau} \sum_{R \in HD_{1}(Q)}
\nu(1.5 B_{R}).
\end{equation}

Define
$$
\Psi_{Q,\nabla} := \sum_{R \in HD_{1}(Q)} g_{R,\nabla} \qquad and \qquad
\Psi_{Q,V} := \sum_{R \in HD_{1}(Q)} g_{R,V}.
$$

\begin{cor}
  \label{cor:PsiQV}
  We have
  $$
\norm{\Psi_{Q,V}}_{L^{1}\br{\mathcal{L}^{d}}} \lesssim
\norm{V}^{\frac{1}{2}}_{L^{d}\br{2 B_{Q}}} \cdot \mu(Q)
  $$
\end{cor}

\begin{prf}  
  Lemma \ref{lem:L1gR} implies that
  \begin{align*}\begin{split}  
 \norm{\Psi_{Q,V}}_{L^{1}\br{\mathcal{L}^{d}}} &\lesssim \sum_{R \in
   HD_{1}(Q)} \norm{V}^{\frac{1}{2}}_{L^{d}(2 B_{R})} \mu \br{R} \\
 &\leq \norm{V}^{\frac{1}{2}}_{L^{d}\br{2 B_{Q}}} \sum_{R \in
   HD_{1}(Q)} \mu (R) \\
 &\leq \norm{V}^{\frac{1}{2}}_{L^{d}\br{2 B_{Q}}} \mu (Q).
 \end{split}\end{align*}
\end{prf}

\begin{lem}
  \label{lem:TV}
  We have
  $$
  \int \abs{S \nu}^{2} \abs{\Psi_{Q,V}}  d \mathcal{L}^{d}
  \lesssim l(Q)^{\epsilon}\norm{V}_{L^{d}\br{10 B_{Q}}}^{\frac{3}{2}} \mu(Q)
  $$
  for some $\epsilon > 0$ that depends only on the dimension.
\end{lem}

\begin{prf}   
The estimate will be proved by obtaining a pointwise estimate of
$S \nu(x)$ on
$\supp  \Psi_{Q,V} \subset 2 B_{Q}$. Fix $x \in 2
B_{Q}$. Since the support of the measure $\nu$ is contained in $Q$,
$$
\abs{S \nu(x)} \leq \int V^{\frac{1}{2}}(y) \abs{\mathcal{E}(x,y)}
d\nu(y) \leq \int_{Q} V^{\frac{1}{2}}(y) \abs{\mathcal{E}(x,y)} d \nu(y).
$$
It is easy to see that $Q \subset 2 B_{Q}  \subset
\tilde{B} := B(x,112r(Q))$. Therefore
$$
 \abs{S \nu(x)} \leq \int_{\tilde{B}} V^{\frac{1}{2}}(y) \abs{\mathcal{E}(x,y)} \, d \nu(y).
$$
H\"{o}lder's inequality and Theorem \ref{thm:MaybSharp} then gives,
\begin{align*}\begin{split}  
 \abs{S \nu(x)} &\lesssim \norm{V}^{\frac{1}{2}}_{L^{d}(\tilde{B})}
 \br{\int_{\tilde{B}} \abs{\mathcal{E}(x,y)}^{\frac{2 d}{2 d -
       1}} d \nu(y)}^{\frac{2d  - 1}{2d}} \\
 &\lesssim \norm{V}_{L^{d}\br{10 B_{Q}}}^{\frac{1}{2}}
 \br{\int_{\tilde{B}} \frac{d \nu(y)}{\abs{x -
       y}^{\frac{(d-2)(2d)}{(2d - 1)}}}}^{\frac{2d - 1}{2d}}.
 \end{split}\end{align*}
Decomposing into annuli,
\begin{align*}\begin{split}  
  \int_{\tilde{B}} \frac{d \nu(y)}{\abs{x - y}^{\frac{(d -
       2)(2 d)}{2d - 1}}} &= \sum_{k = 0 }^{\infty} \int_{2^{-k} \tilde{B}
   \setminus 2^{-k-1} \tilde{B}} \frac{d \nu(y)}{\abs{x - y}^{\frac{(d
       - 2)(2d)}{2d - 1}}} \\
 &\lesssim \sum_{k = 0}^{\infty} \frac{\nu(2^{-k} \tilde{B})}{\br{2^{-k-1}
   r(Q)}^{\frac{(d-2)(2d)}{2d-1}}} \\
 &\lesssim \sum_{k = 0}^{\infty}
 \frac{\br{2^{-k}r(Q)}^{d - 1}}{\br{2^{-k-1}r(Q)}^{\frac{(d - 2)(2d)}{2d-1}}}.
 \end{split}\end{align*}
Since $d - 1 > \frac{(d-2)(2d)}{2d-1}$ we then obtain
$$
 \int_{\tilde{B}} \frac{d \nu(y)}{\abs{x - y}^{\frac{(d -
       2)(2d)}{2d-1}}} \lesssim r(Q)^{\epsilon_{1}} \lesssim l(Q)^{\epsilon_{1}}
 $$
 for some $\epsilon_{1} > 0$ that only depends on $d$. This gives the
 pointwise bound
 $$
\abs{S\nu(x)} \lesssim \norm{V}_{L^{d}\br{10 B_{Q}}}^{\frac{1}{2}} l(Q)^{\epsilon/2}
$$
for some $\epsilon > 0$ that depends only on the dimension. Corollary
\ref{cor:PsiQV} then gives
\begin{align*}\begin{split}  
 \int \abs{S\nu(x)}^{2} \abs{\Psi_{Q,V}} \, d \mathcal{L}^{d}
 &\lesssim l(Q)^{ \epsilon} \norm{V}_{L^{d}\br{10 B_{Q}}} \int \abs{\Psi_{Q,V}} \, d
 \mathcal{L}^{d} \\
 &\lesssim l(Q)^{ \epsilon} \norm{V}_{L^{d}\br{10 B_{Q}}}^{\frac{3}{2}} \mu(Q).
 \end{split}\end{align*}
 \end{prf}

To prove the following lemma we will need the assumption that $T$ is
bounded, the decay of the function $g_{R,\nabla}$, the size condition from Theorem \ref{thm:Size} and the smoothness conditions from Corollary \ref{cor:Reg}, which are available thanks to Proposition \ref{prop:AdmissibleDM} . Since with these ingredients the proof works in an identical manner to the
corresponding potential-free statement {\cite[Lemma~6.1]{conde2019failure}}, we will not include it here. 

\begin{lem}
  \label{lem:6.1}
  The estimate
  $$
\int \abs{T \br{\abs{\Psi_{Q,\nabla}} \mathcal{L}^{d}}}^{2} \, d \nu
\lesssim \mu(Q)
$$
holds.
\end{lem}

 We are now at a point where the contradiction argument can be completed.
On successively applying \eqref{eqtn:HD1} and the reproducing
formula \eqref{eqtn:Reproducing},
\begin{align*}\begin{split}  
    \mu(Q) &\simeq \sum_{R \in HD_{1}(Q)} \nu(1.5B_{R}) \\
    &\leq \sum_{R \in HD_{1}(Q)} \int \varphi_{R} \, d \nu \\
    &= \sum_{R \in HD_{1}(Q)} \int T^{*} \br{g_{R,\nabla} \, d
      \mathcal{L}^{d}} d \nu + \int S^{*} \br{g_{R,V} \, d
      \mathcal{L}^{d}} d \nu \\
    &= \int T \nu \cdot \Psi_{Q,\nabla} \, d \mathcal{L}^{d}
    + \int S \nu \cdot \Psi_{Q, V} \, d \mathcal{L}^{d}.
  \end{split}\end{align*}
The term corresponding to $T$ can be treated in an identical
manner to \cite{conde2019failure} by applying Lemma \ref{lem:GlobalPointwise} to obtain
$$
\int T \nu \cdot \Psi_{Q,\nabla} \, d \mathcal{L}^{d} \lesssim
\br{\lambda + l(Q)^{\gamma}}^{\frac{1}{4}} \mu(Q).
$$
For the term corresponding to $S$, apply Corollary \ref{cor:PsiQV}
and Lemma \ref{lem:TV} to obtain
\begin{align*}\begin{split}  
 \int S \nu \cdot \Psi_{Q,V} \, d \mathcal{L}^{d} &\leq \br{\int
 \abs{S \nu(x)}^{2} \abs{\Psi_{Q,V}} \, d \mathcal{L}^{d}}^{\frac{1}{2}} \br{\int \abs{\Psi_{Q,V}}\, d
 \mathcal{L}^{d}}^{\frac{1}{2}} \\
&\lesssim l(Q)^{\epsilon/2} \norm{V}_{L^{d}(10B_{Q})} \mu(Q). 
 \end{split}\end{align*}
Putting these two estimates together gives
$$
\mu(Q) \lesssim \br{\br{\lambda + l(Q)^{\gamma}}^{\frac{1}{4}} +
  l(Q)^{\epsilon/2} \norm{V}_{L^{d}\br{10B_{Q}}}}\mu(Q).
$$
For $Q$ and $\lambda$ small enough this will result in a contradiction.

\subsection{A Shorter Proof of Proposition \ref{prop:SecondReduction}}
\label{sec:ShortProof}

In this section, we will demonstrate an alternative proof to Proposition
\ref{prop:SecondReduction} that replaces Section
\ref{subsec:Contradiction}. This alternative argument instead relies
upon the potential-free result from \cite{conde2019failure}. Let
$T^{0}_{A}$ denote the operator
$$
T^{0}_{A}\omega(x) := \int \nabla_{1} \mathcal{E}^{0}_{A}(x,y) \, d \omega(y)
$$
for a measure $\omega$ on $\R^{d}$, where if you recall
$\mathcal{E}^{0}_{A}$ is the fundamental solution of $L_{A}^{0} = -
\mathrm{div} A \nabla$ on $\R^{d}$. Also define
$$
R \omega (x) := \int \br{\nabla_{1} \mathcal{E}(x,y) - \nabla_{1}
  \mathcal{E}^{0}_{A}(x,y)} \, d \omega(y).
$$
It was proved in {\cite[Section~6]{conde2019failure}} that there must
exist some $C > 0$ such that
$$
\norm{T^{0}_{A} \sigma}_{L^{2}\br{\sigma}} \geq C \cdot \sigma\br{\R^{d}}^{\frac{1}{2}}
$$
for all sufficiently small cubes. Therefore,
\begin{align*}\begin{split}  
 \norm{T \sigma}_{L^{2}\br{\sigma}} &\geq
\norm{T^{0}_{A} \sigma}_{L^{2}\br{\sigma}} - \norm{R
  \sigma}_{L^{2}\br{\sigma}} \\
&\geq C \cdot \sigma \br{\R^{d}}^{\frac{1}{2}} - \norm{R \sigma}_{L^{2}\br{\sigma}}.
\end{split}\end{align*}
To prove Proposition \ref{prop:SecondReduction} it is therefore
sufficient to prove that there exists $N_{0} > 0$ such that
\begin{equation}
  \label{eqtn:ShorterProof}
\norm{R \sigma}_{L^{2}\br{\sigma}} \leq \frac{C}{2} \sigma \br{\R^{d}}^{\frac{1}{2}}
\end{equation}
for all $Q \in \Sigma_{N}$ with $N > N_{0}$. Recall that $Q$ will be
admissible in the sense that $\abs{x - y}
\leq \inf_{\supp  \mu} \rho$ for all $x, \, y \in Q$ if we
let $N_{0} > 0$ be large enough so that $Q \in \Sigma_{N}$ with $N >
N_{0}$ implies that $Q \in \mathcal{D}_{k}$ for some $k \geq K_{0}$,
where $K_{0}$ is as given in Proposition \ref{prop:AdmissibleDM}.
Since the measure $\sigma$ is localized to the cube $Q$, Proposition
\ref{prop:Flatness} can then be
applied. For $x \in \supp  \sigma$, this will give
\begin{align*}\begin{split}  
 \abs{R \sigma(x)} &\leq \int_{Q} \abs{\nabla_1 \mathcal{E}(x,y) -
   \nabla_1 \mathcal{E}^{0}_{A}(x,y)} \, d \sigma(y) \\
 &\lesssim \int_{B(x, 56 r(Q))} \frac{d \sigma(y)}{\abs{x - y}^{d - 1 - \beta}} \\
 &= \sum_{k = 0}^{\infty} \int_{2^{-k} B (x, 56 r(Q)) \setminus 2^{-(k + 1)}
   B(x,56 r(Q))} \frac{d \sigma(y)}{\abs{x - y}^{d - 1 - \beta}} \\
 &\lesssim \sum_{k = 0}^{\infty} \frac{\sigma \br{2^{-k}
     B (x, 56 r(Q))}}{\br{2^{-(k + 1)}r(Q)}^{d - 1 - \beta}}.
\end{split}\end{align*}
The polynomial growth of the measure $\sigma$ then implies
$$
 \abs{R \sigma(x)} \lesssim \sum_{k = 0}^{\infty}
 \br{2^{-k}r(Q)}^{\beta} \lesssim r(Q)^{\beta}.
$$
This pointwise estimate leads to
$$
\norm{R \sigma}_{L^{2}\br{\sigma}}^{2} = \int_{Q} \abs{R
   \sigma(x)}^{2} d \sigma(x) 
 \lesssim r(Q)^{2\beta} \sigma\br{\R^{d}}.
$$
Clearly if we set $N_{0}$ large enough then \eqref{eqtn:ShorterProof}
will be satisfied for all $Q \in \Sigma_{N}$ with $N > N_{0}$. This
completes our second proof of Proposition \ref{prop:SecondReduction}.


\vspace*{0.2in}

\hspace*{-0.2in} \begin{minipage}[b]{0.37\linewidth}
    \textbf{Julian Bailey} \\
   u4545137@alumni.anu.edu.au 
  \end{minipage}  \begin{minipage}[b]{0.31\linewidth}
    \textbf{Andrew J. Morris} \\
    a.morris.2@bham.ac.uk 
  \end{minipage} 
  \begin{minipage}[b]{0.33\linewidth}
    \textbf{Maria Carmen Reguera} \\
    m.reguera@bham.ac.uk
\end{minipage}

\vspace*{0.1in}

\small \textsc{School of Mathematics, University of Birmingham, Edgbaston,
  Birmingham, B15 2TT, UK}

\end{document}